\documentclass{jssc}

\def\textsubscript#1%
{$_{\text{#1}}$}

\def\cdd{\mbox{\boldmath$\cdot$}~}

\input{amssym.def}
\include{graphix}

\usepackage{lineno}
\usepackage[colorlinks=true,breaklinks=true,pdftex]{hyperref}
\modulolinenumbers[1]

\usepackage[dvipsnames]{xcolor}
\usepackage[ruled, linesnumbered]{algorithm2e}
\usepackage{algorithmic}
\usepackage[english]{babel}
\usepackage{bm}
\usepackage{placeins}
\usepackage{xfrac}
\usepackage{epstopdf,cmap,amssymb,amsfonts,amsmath,mathtext,enumerate,float,natbib,indentfirst,graphicx,multirow,color,setspace,subfigure,amsthm,chngcntr,url}
\usepackage{lmodern}
\usepackage{mathrsfs}
\usepackage[colorlinks=true, breaklinks=true, pdftex]{hyperref}
\usepackage{booktabs}

\makeatletter
\def\@oddfoot{\hfill}
\newcount\shumeicount
\def\setshumei#1#2#3{%
  \shumeicount=\count0
  \def\@oddhead{%
    \raise-5pt\hbox to0pt{\vrule width\hsize height 0pt depth 0.4pt\hss}\relax
    \ifnum \shumeicount=\count0
      \raise-7pt\hbox to0pt{\vrule width\hsize height 0pt depth 0.4pt\hss}\relax
      #1
    \else
      \ifodd\count0
        #2
      \else
        #3
       \fi
     \fi
  }%
}
\makeatother
\makeatletter
\def\@oddfoot{\hfill}
\newcount\shujiaocount
\def\setshujiao{%
  \shujiaocount=\count0
  \def\@oddfoot{%
      \ifodd\count0
      \else
      \fi
  }%
}
\makeatother
\def\title#1#2#3#4{{
  \vspace*{0.3cm}
  \begin{flushleft} \Large\bf #1\end{flushleft}
  \vspace*{-0.2cm}
      \begin{flushleft}
      \bf #2
      \end{flushleft}
      \footnotetext{\hspace{-6mm} #3\\ #4}}}

\def\dshm#1#2#3#4
{\setshumei{J Syst Sci Complex (#1) #2:
{\thepage--\pageref{LastPage}}\hfill}
            {\hfill {\small #3}\hfill\hbox to0pt{\hss\thepage}}
            {\hbox to0pt{\thepage\hss}\hfill {\small #4}\hfill
            }
            \setshujiao}
\def\drd#1#2
{{\vskip 1cm\small \begin{flushleft}
 #1 \\
 #2\\
\copyright The Editorial Office of JSSC \&  Springer-Verlag GmbH
Germany 2018
\end{flushleft}}}

\def\bar{\overline}
\def\epsilon{\varepsilon}

\begin{document}


\title{Bijectivity analysis of rational T-spline surfaces via Bernstein representations}
{\uppercase{Li} Jia-Xuan$^\diamond$ \cdd \uppercase{Yu} Ying-Ying$^\diamond$ \cdd \uppercase{Liu} Ya-Shu \cdd \uppercase{Li} Xin \cdd \uppercase{Ji} Ye$^*$ \cdd \uppercase{Zhu} Chun-Gang}
{\uppercase{Li} Jia-Xuan \cdd \uppercase{Yu} Ying-Ying \cdd \uppercase{Liu} Ya-Shu \\
School of Mathematics, Liaoning Normal University, Dalian $116029$, China. \\
\uppercase{Li} Xin \cdd \uppercase{Zhu} Chun-Gang \\
School of Mathematical Sciences, Dalian University of Technology, Dalian $116024$, China. \\
\uppercase{Ji} Ye \\
Delft Institute of Applied Mathematics, Delft University of Technology, Delft $2628$ CD, the Netherlands. Email: y.ji-1@tudelft.nl \\
} 
{$^*$Corresponding author.\\
$^\diamond$The two authors contribute equally to this work.}

\drd{DOI: }{Received: x x 20xx}{ / Revised: x x 20xx}


\dshm{20XX}{XX}{Bijectivity analysis of rational T-spline surfaces via Bernstein representations}{\uppercase{Li} Jia-Xuan$^\diamond$ \cdd \uppercase{Yu} Ying-Ying$^\diamond$ \cdd \uppercase{Liu} Ya-Shu \cdd \uppercase{Li} Xin \cdd \uppercase{Ji} Ye$^*$ \cdd \uppercase{Zhu} Chun-Gang}

\Abstract{Ensuring the bijectivity of spline-based parameterizations is fundamental in geometric modeling and isogeometric analysis, as invalid mappings may lead to self-intersections, singular Jacobians, and numerical instability. While T-splines offer enhanced flexibility through local refinement, this flexibility also makes bijectivity verification significantly more challenging. In this work, we propose a rigorous and efficient framework for bijectivity analysis of rational T-spline surfaces based on B\'ezier extraction. The key idea is to reformulate the T-spline representation into a collection of element-wise rational B\'ezier patches, on which the Gram determinant of the mapping admits a Bernstein polynomial representation. This enables a coefficient-based analysis of local regularity by exploiting the convex hull and positivity properties of the Bernstein basis. Based on this formulation, we derive a sufficient condition for bijectivity from the nonnegativity of Bernstein coefficients, together with a necessary condition based on the sign consistency of corner coefficients. For cases where these conditions are inconclusive, we introduce a hierarchical subdivision strategy that progressively localizes ambiguous regions and resolves them through refinement. The proposed method provides a certified and adaptive procedure for bijectivity verification that avoids dense numerical sampling and remains computationally efficient. Numerical experiments on complex T-spline geometries demonstrate that the approach accurately detects both valid and near-degenerate configurations, while scaling effectively to large models with thousands of rational B\'ezier patches. The framework is fully compatible with standard isogeometric analysis workflows.}      

\Keywords{Isogeometric analysis, T-splines, bijectivity analysis, Bernstein polynomials, B\'ezier extraction}        




\section{Introduction}
\label{sec:introduction}

Spline-based representations, such as B-splines and Non-Uniform Rational B-Splines (NURBS)~\citep{piegl2012nurbs}, have long been the standard in computer-aided geometric design (CAGD) due to their strong approximation properties and ability to represent complex geometries exactly. Within the framework of isogeometric analysis (IGA)~\citep{hughes2005isogeometric,cottrell2009isogeometric}, these representations further enable a seamless integration between geometric modeling and numerical simulation~\citep{xu2017unified}. However, classical tensor-product constructions impose a rigid topological structure, which limits their flexibility in handling local refinement and complex geometries.

T-splines, introduced in~\citep{sederberg2003t}, generalize NURBS by allowing T-junctions in the control mesh. This relaxation of the tensor-product constraint enables local refinement without global mesh propagation, making T-splines particularly attractive for adaptive analysis and complex geometric modeling~\citep{sederberg2004tspline,sederberg2008watertight}. As a result, T-splines have been widely adopted in both CAGD and IGA~\citep{verhoosel2011isogeometric,verhoosel2011isogeometricgradient,benson2010generalized,dorfel2010adaptive}.

Despite these advantages, the increased flexibility of T-splines significantly complicates the analysis of geometric validity~\citep{brovka2014new}. A fundamental requirement for both geometric modeling and numerical simulation is the bijectivity of the parametrization~\citep{ji2023improved}. Violation of bijectivity leads to self-intersections or fold-overs~\citep{yu2026regularity,Coons-surface}, resulting in invalid geometries and failure of numerical procedures such as quadrature, stiffness matrix assembly, and convergence of the solution~\citep{ji2021constructing}. Ensuring bijectivity is therefore a critical yet challenging problem in practical T-spline applications.

Existing approaches for validity checking typically rely on numerical sampling, Jacobian sign evaluation, or optimization-based techniques~\citep{terahara2025t,terahara2025t2}. However, these approaches face inherent limitations. Sampling-based methods cannot guarantee global bijectivity, as they only provide pointwise verification, while optimization-based approaches, which aim to compute the global minimum of the Jacobian determinant, are often computationally expensive and difficult to integrate into standard IGA workflows~\citep{lopez2017spline}. These challenges motivate the development of more robust and theoretically grounded verification methods~\citep{ji2023curvature}.

In this work, we propose a Bernstein-based framework for bijectivity analysis of rational T-spline parametrizations. The key idea is to reformulate the T-spline representation into a collection of element-wise rational B\'ezier patches via B\'ezier extraction, on which the Gram determinant of the mapping admits a Bernstein polynomial representation. This formulation enables a coefficient-based analysis of local regularity by exploiting the convex hull and positivity properties of the Bernstein basis.

Based on this representation, we derive both sufficient and necessary conditions for bijectivity. A sufficient condition is obtained from the nonnegativity of Bernstein coefficients, while a necessary condition is derived from the sign consistency of corner coefficients. For cases where these conditions are inconclusive, we introduce a hierarchical subdivision strategy that progressively localizes ambiguous regions and resolves them through refinement.

The main contributions of this work are summarized as follows:
\begin{itemize}
    \item A Bernstein-based formulation of the Gram determinant for rational T-spline parametrizations via B\'ezier extraction;
    \item Sufficient and necessary sign conditions for bijectivity that enable efficient certification and early rejection;
    \item A hierarchical subdivision strategy that resolves ambiguous cases through localized refinement;
    \item An efficient and scalable algorithm fully compatible with standard isogeometric analysis workflows.
\end{itemize}

These ingredients are integrated into a practical algorithm that performs element-wise bijectivity verification. The resulting method provides strong theoretical guarantees while remaining computationally efficient. Compared with existing approaches, the proposed framework offers several advantages. First, it establishes a mathematically rigorous criterion based on Bernstein representations, avoiding reliance on dense numerical sampling. Second, it is fully compatible with standard B\'ezier extraction procedures in isogeometric analysis, making it straightforward to integrate into existing workflows. Finally, the hierarchical subdivision strategy ensures robustness in near-degenerate configurations by localizing refinement to ambiguous regions, thereby maintaining overall efficiency.

The remainder of this paper is organized as follows. Section~\ref{sec:prelimiaries} introduces the T-spline representation and the B\'ezier extraction framework underlying the proposed analysis. Section~\ref{sec:theory} develops the Bernstein-based formulation of the Gram determinant and establishes the corresponding sign conditions for bijectivity. Section~\ref{sec:numerical} presents numerical experiments to validate the effectiveness and efficiency of the method. Section~\ref{sec:conclusion} concludes the paper.

\section{Preliminaries}
\label{sec:prelimiaries}

This section briefly reviews the T-spline representation and the B\'ezier extraction framework, which provide the foundation for the Bernstein-based analysis of the Gram determinant developed in the following sections.

\subsection{T-spline representation}
\label{sec:tspline}

NURBS are the de facto standard in computer-aided geometric design due to their ability to represent complex geometries exactly and their favorable approximation properties~\citep{piegl2012nurbs}. However, classical NURBS constructions rely on tensor-product parameterizations, which introduce strong global coupling among control points and limit the flexibility of local refinement (see Fig.~\ref{fig:t_spline}).

\begin{figure}
    \centering
    \includegraphics[width=0.8\linewidth]{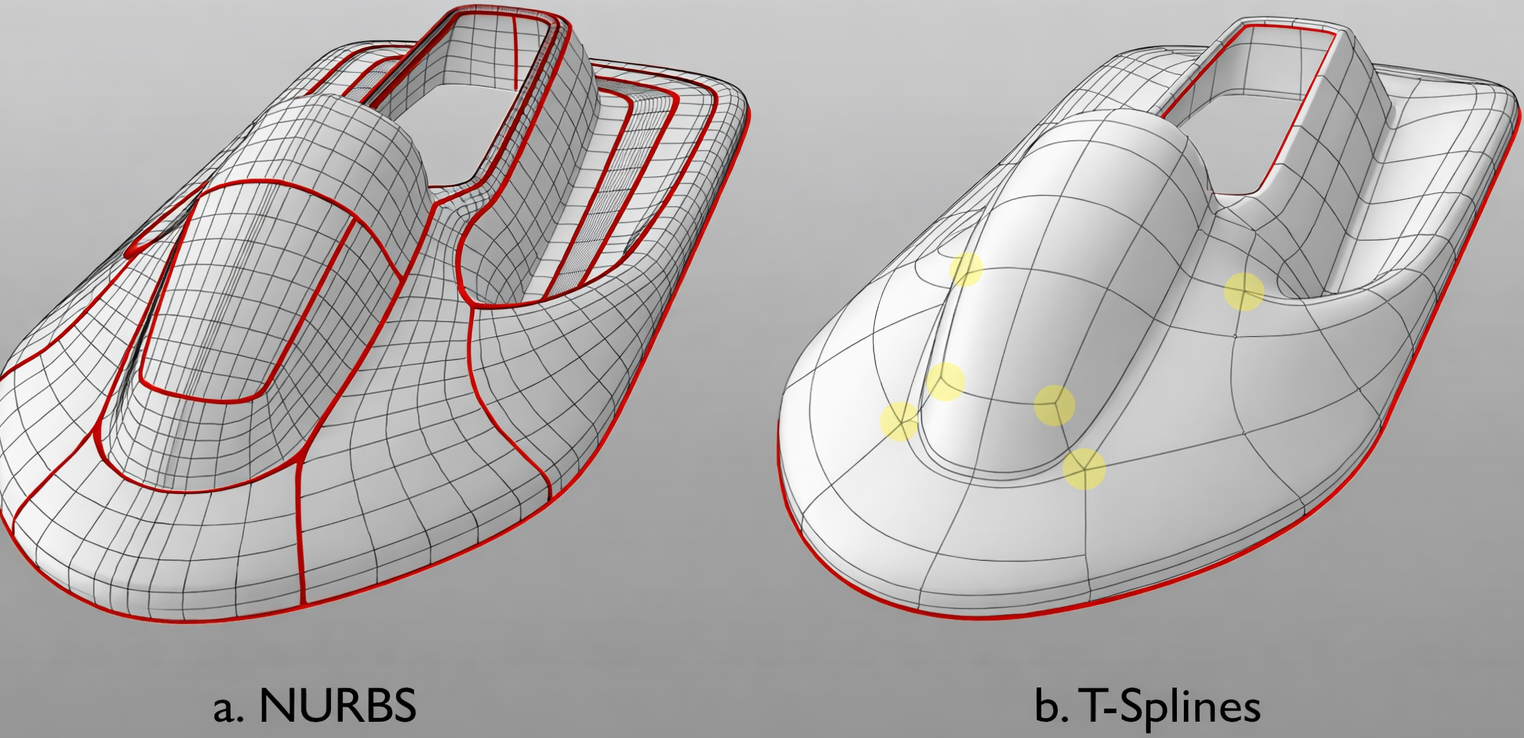}
    \caption{Comparison of NURBS and T-spline control meshes (adapted from~\citep{sederberg2010t}). In tensor-product NURBS, refinement requires global propagation of control points, whereas T-splines allow local refinement through the introduction of T-junctions.}
    \label{fig:t_spline}
\end{figure}

T-splines generalize NURBS by relaxing the tensor-product constraint through the introduction of T-junctions in the control mesh~\citep{sederberg2003t}, as shown in Fig.~\ref{fig:t_spline}. This construction enables local refinement without global propagation of control points, making T-splines particularly suitable for complex geometries and adaptive isogeometric analysis. Under suitable admissibility conditions, such as analysis-suitable T-splines~\citep{scott2012local}, the resulting basis functions are linearly independent, form a partition of unity, and possess the smoothness required for analysis.

Let $(\xi,\eta) \in \Omega \subset \mathbb{R}^2$ denote the parametric domain associated with a T-mesh, which induces a partition of $\Omega$ into a collection of elements. A rational T-spline surface is defined as
\begin{equation}
\bm{x}(\xi,\eta)
=
\sum_{i=1}^{n}
R_i(\xi,\eta)\,\mathbf{P}_i,
\label{eq:rational_t_spline_surface}
\end{equation}
where $\mathbf{P}_i \in \mathbb{R}^d$ are the control points, and the rational basis functions are given by
\begin{equation}
R_i(\xi,\eta)
=
\frac{N_i(\xi,\eta)\, w_i}
{\displaystyle \sum_{j=1}^{n} N_j(\xi,\eta)\, w_j},
\label{eq:rational_t_spline_basis}
\end{equation}
with $w_i > 0$ denoting the associated weights. Here, $N_i(\xi,\eta)$ are locally supported T-spline blending functions constructed from local knot vectors induced by the T-mesh topology.

From an algebraic perspective, T-splines share the same rational structure as NURBS~\citep{sederberg2004tspline}. The key distinction lies in the use of local knot vectors: while tensor-product NURBS are defined with respect to global knot vectors, T-spline basis functions are associated with local knot intervals, enabling flexible and localized refinement~\citep{sederberg2008watertight}. This locality naturally leads to element-wise representations, which form the basis for the Bernstein-based analysis developed in the following section.

\subsection{B\'ezier extraction and Bernstein representation}
\label{sec:bezier_extraction}

B\'ezier extraction establishes a linear transformation between global T-spline basis functions and element-wise Bernstein basis functions defined on each element of the parametric domain. In particular, the restriction of the T-spline basis to an element can be represented as a linear combination of Bernstein polynomials via an element-specific extraction operator. This transformation enables the use of standard finite element assembly procedures within the isogeometric analysis framework by expressing spline basis functions in Bernstein form on each element~\citep{borden2011isogeometric,scott2011isogeometric}.

Beyond its computational advantages, this element-wise Bernstein representation provides a convenient framework for subsequent analysis, as it allows geometric quantities, such as the Jacobian determinant, to be expressed directly in terms of Bernstein polynomials.

\begin{figure}
    \centering
    \includegraphics[width=0.95\linewidth]{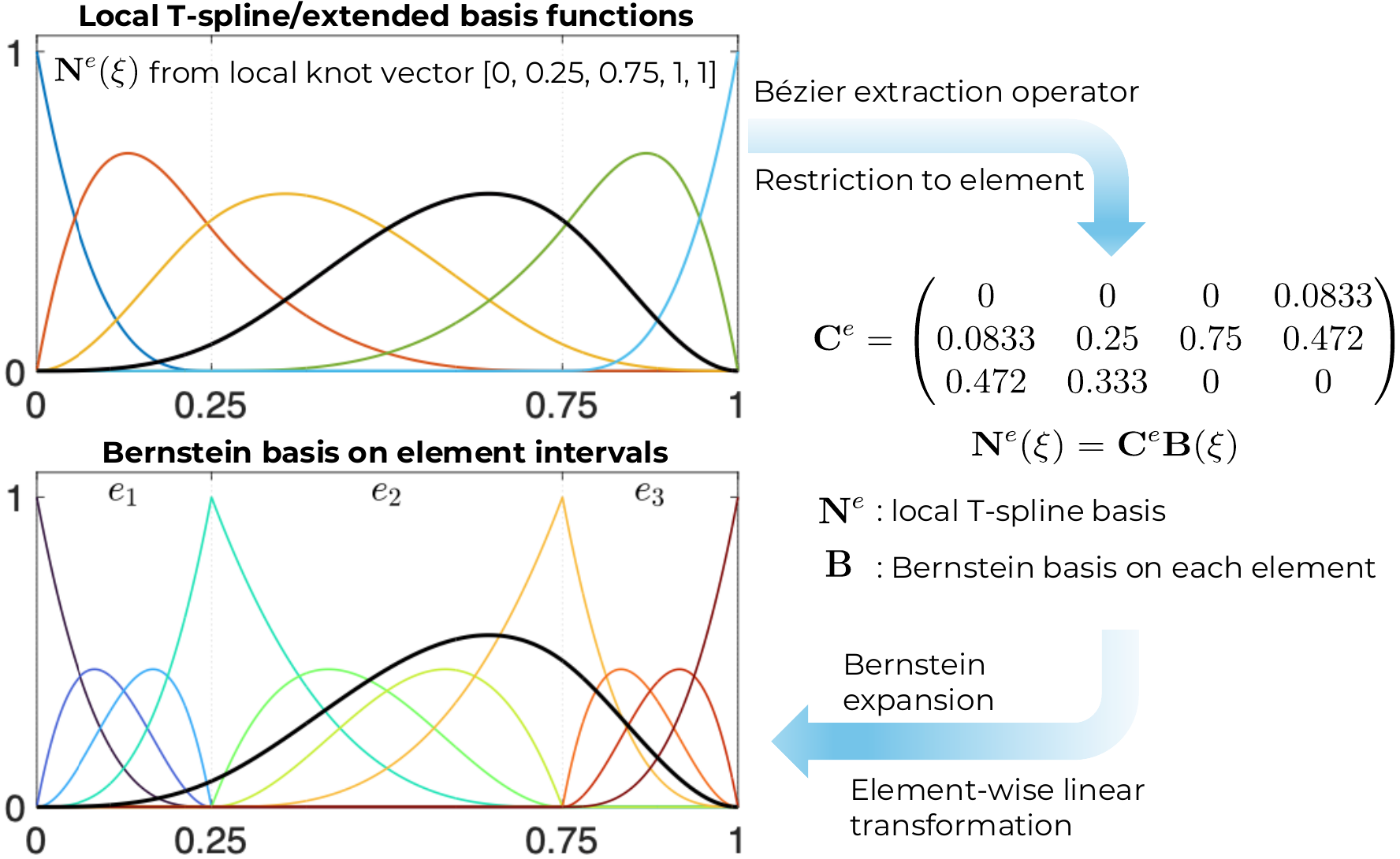}
    \caption{B\'ezier extraction in T-splines. The local T-spline basis functions $\mathbf{N}^e$ are first restricted to an element $e$, where they are represented as a linear combination of Bernstein basis functions $\mathbf{B}$ through the extraction operator $\mathbf{C}^e$, i.e., $\mathbf{N}^e = \mathbf{C}^e \mathbf{B}$. This element-wise Bernstein representation enables localized analysis while preserving exact geometric equivalence with the original T-spline parameterization.}
    \label{fig:bezier_extraction}
\end{figure}

Let $\mathbf{N}^e$ denote the vector of T-spline basis functions that are nonzero over an element $e$ (see Fig.~\ref{fig:bezier_extraction}). Restricted to the element $e$, these basis functions admit the representation
\begin{equation}
    \mathbf{N}^e = \mathbf{C}^e \mathbf{B},
\end{equation}
where $\mathbf{B}$ is the vector of Bernstein basis functions defined on the reference domain of $e$, and $\mathbf{C}^e$ is the corresponding B\'ezier extraction operator, which depends only on the local T-mesh topology and remains constant within the element.

For a rational T-spline surface~\eqref{eq:rational_t_spline_surface}, the mapping restricted to an element $e$ can be written as
\begin{equation}
    \bm{x}^e = (\mathbf{P}^e)^{\top} \mathbf{R}^e,
\end{equation}
where $\mathbf{P}^e$ collects the control points associated with the element, and $\mathbf{R}^e$ denotes the vector of rational basis functions. 
Using the extraction operator, the rational basis functions take the form
\begin{equation}
\mathbf{R}^e
=
\frac{\mathrm{diag}(\mathbf{w}^e)\,\mathbf{C}^e\,\mathbf{B}}{W},
\qquad
W = (\mathbf{w}^e)^{\top} \mathbf{C}^e \mathbf{B},
\end{equation}
where $\mathbf{w}^e$ is the vector of weights associated with the element, and $\mathrm{diag}(\mathbf{w}^e)$ is the corresponding diagonal matrix.

The first derivatives with respect to $\alpha \in \{\xi,\eta\}$ are given by
\begin{equation}
\mathbf{R}^e_{,\alpha}
=
\mathrm{diag}(\mathbf{w}^e) \mathbf{C}^e
\left(
\frac{\mathbf{B}_{,\alpha}}{W}
-
\frac{W_{,\alpha}}{W^2}\mathbf{B}
\right),
\end{equation}
where
\begin{equation}
W_{,\alpha}
=
(\mathbf{w}^e)^{\top} \mathbf{C}^e \mathbf{B}_{,\alpha}.
\end{equation}

Higher-order derivatives can be derived analogously. Consequently, all geometric quantities associated with the mapping can be expressed element-wise in Bernstein form prior to the rational weighting, which is particularly advantageous for subsequent analysis.

A key advantage of this formulation lies in the structural properties of the Bernstein basis. Bernstein polynomials are nonnegative over the reference domain and satisfy the convex hull property. Consequently, any quantity expressed in Bernstein form can be bounded directly in terms of its coefficients.

This observation is central to the present work. In particular, the Gram determinant $\det(\bm{\mathcal{J}}^\top \bm{\mathcal{J}})$ admits a Bernstein representation on each element. Therefore, its positivity can be assessed through the sign of the corresponding Bernstein coefficients. In particular, nonnegativity of the coefficients provides a sufficient condition for positivity of the Gram determinant over the element, and thus guarantees local bijectivity of the mapping. This coefficient-based perspective forms the foundation of the bijectivity analysis developed in the subsequent sections.

\section{Bernstein-coefficient-based bijectivity analysis}
\label{sec:theory}

In this section, we develop a Bernstein-coefficient-based framework for the bijectivity analysis of rational T-spline parametrizations. By means of B\'ezier extraction, the global T-spline surface is decomposed into a collection of element-wise rational B\'ezier patches, which enables a localized analysis of the Jacobian and its Gram determinant in Bernstein form.

Building on this representation, we express the Gram determinant on each element as a Bernstein polynomial and analyze its positivity through the corresponding coefficients. This formulation allows us to derive both sufficient and necessary conditions for local bijectivity, as well as a hierarchical strategy for resolving ambiguous cases.

\subsection{Element-wise rational B\'ezier representation}

Let $\bm{x}(\xi,\eta)$ be a rational T-spline surface defined over a parametric domain $\Omega$. Through B\'ezier extraction, the surface can be decomposed into a collection of rational B\'ezier elements. On each element $e$, the mapping admits the representation
\begin{equation}
\bm{x}^e(\xi,\eta)
=
\frac{\mathbf{F}^e(\xi,\eta)}{W^e(\xi,\eta)},
\qquad
(\xi,\eta)\in[0,1]^2,
\end{equation}
where $\mathbf{F}^e:\,[0,1]^2\to\mathbb{R}^3$ and $W^e:\,[0,1]^2\to\mathbb{R}$ are polynomial functions expressed in the Bernstein basis.

For notational convenience, define
\begin{equation}
\mathbf{U}^e = \mathbf{F}^e_\xi W^e - \mathbf{F}^e W^e_\xi,
\qquad
\mathbf{V}^e = \mathbf{F}^e_\eta W^e - \mathbf{F}^e W^e_\eta.
\end{equation}
Then, the Jacobian matrix of the parametrization is given by
\begin{equation}
\bm{\mathcal{J}}^e
=
[\bm{x}^e_\xi \ \bm{x}^e_\eta]
=
\frac{1}{(W^e)^2}
[\mathbf{U}^e \ \mathbf{V}^e].
\end{equation}

In the following, all bijectivity analysis is conducted at the element level.

\subsection{Bernstein representation of the Gram determinant}

For a rational B\'ezier patch, the Gram determinant is a highly nonlinear rational function whose sign is difficult to assess directly. To overcome this difficulty, we seek to express it in the Bernstein basis, which enables coefficient-based analysis due to the convex hull property of Bernstein polynomials. More specifically, once a Bernstein expansion is obtained, the sign of the Gram determinant can be inferred from its coefficients, thereby providing a practical and efficient criterion for regularity verification. To this end, we establish the following result.

\begin{theorem}[Bernstein representation of the Gram determinant]
\label{thm:gram-bernstein}
Let $\bm{\mathcal{J}}^e(\xi,\eta)$ be the Jacobian matrix of a rational B\'ezier element $e$ of bi-degree $(p,q)$. Then the Gram determinant admits the Bernstein expansion
\begin{equation}
\det\!\left((\bm{\mathcal{J}}^e)^\top \bm{\mathcal{J}}^e \right)(\xi,\eta)
=
\frac{1}{(W^e(\xi,\eta))^{8}}
\sum_{r=0}^{8p-2}\sum_{s=0}^{8q-2}
D_{rs}^e \, B_r^{8p-2}(\xi) B_s^{8q-2}(\eta),
\label{eq:JeTJe-final}
\end{equation}
where the coefficients $D_{rs}^e\in\mathbb{R}$ depend only on the control points and weights associated with the element.
\end{theorem}

\begin{proof}
The result follows from the closure properties of the Bernstein basis under differentiation, multiplication, and linear combinations.

First, the rational derivatives can be written as
\[
\bm{x}^e_\xi = \frac{\mathbf{U}^e}{(W^e)^2}, \quad
\bm{x}^e_\eta = \frac{\mathbf{V}^e}{(W^e)^2},
\]
where $\mathbf{U}^e$ and $\mathbf{V}^e$ are polynomial vector fields admitting Bernstein representations.

Second, using the product rule for Bernstein polynomials, both $\mathbf{U}^e$ and $\mathbf{V}^e$ can be expressed in tensor-product Bernstein bases of bi-degrees $(2p-1,2q)$ and $(2p,2q-1)$, respectively.

Third, the Gram determinant can be written as
\[
\det((\bm{\mathcal{J}}^e)^\top \bm{\mathcal{J}}^e)
=
\frac{1}{(W^e)^8}
\left[
(\mathbf{U}^e\cdot\mathbf{U}^e)
(\mathbf{V}^e\cdot\mathbf{V}^e)
-
(\mathbf{U}^e\cdot\mathbf{V}^e)^2
\right],
\]
and each term remains in a Bernstein space due to closure under inner products and products.

Therefore, the Gram determinant admits a Bernstein expansion of bi-degree $(8p-2,8q-2)$.

The detailed coefficient construction is provided in Appendix \ref{sec:appendix_A}.
\end{proof}

\subsection{A sufficient condition for elementwise regularity}

We begin by recalling two standard results that relate the rank of the Jacobian matrix to local regularity.

\begin{lemma}
Let $A \in \mathbb{R}^{m \times n}$. Then
\begin{equation}
    \operatorname{rank}(A) = \operatorname{rank}(A^T A).
\end{equation}
\end{lemma}

\begin{lemma} \citep{knupp2020fundamentals}
Let $\mathbf{x} : \mathbb{R}^2 \to \mathbb{R}^n$ be a differentiable mapping with Jacobian matrix $\bm{\mathcal{J}}$. If $\bm{\mathcal{J}}$ has full column rank, then $\mathbf{x}$ is locally regular.
\label{lemma1}
\end{lemma}

The following result provides a convenient criterion for verifying local regularity through the Gram determinant.

\begin{theorem}
Let $\mathbf{x}(u,v)$ be a parametric surface with Jacobian matrix
\begin{equation}
    \bm{\mathcal{J}} = (\mathbf{x}_u \ \mathbf{x}_v).
\end{equation}
If
\begin{equation}
    \det(\bm{\mathcal{J}}^\top \bm{\mathcal{J}}) \neq 0,
\end{equation}
then $\mathbf{x}$ is locally regular.
\end{theorem}

\begin{proof}
If $\det(\bm{\mathcal{J}}^\top \bm{\mathcal{J}}) \neq 0$, then $\bm{\mathcal{J}}^\top \bm{\mathcal{J}}$ is nonsingular and hence full rank. 
It follows that $\operatorname{rank}(\bm{\mathcal{J}}^\top \bm{\mathcal{J}}) = 2$. 
By Lemma~\ref{lemma1}, we obtain $\operatorname{rank}(\bm{\mathcal{J}}) = 2$, i.e., $\bm{\mathcal{J}}$ has full column rank. 
Therefore, the mapping is locally regular.
\end{proof}

We now apply this general result to the Bernstein representation of the Gram determinant derived in the previous section. Recall that the elementwise Gram determinant $W^e(\xi,\eta)$ admits a Bernstein expansion with coefficients $D_{rs}^e$. The positivity of $W^e$ can therefore be analyzed through the sign of its Bernstein coefficients.

Without loss of generality, we restrict our discussion to the case where the Bernstein coefficients are nonnegative, as the analysis for the negative case follows analogously.

\begin{theorem}[Sufficient condition for elementwise regularity]
\label{cor:regularity}
Assume that $W^e(\xi,\eta)>0$ on $[0,1]^2$. 
If all Bernstein coefficients in \eqref{eq:bernstein_coefficients} satisfy
\begin{equation}
D_{rs}^e \ge 0,
\qquad
\text{and at least one } D_{rs}^e > 0,
\end{equation}
then
\begin{equation}
\det\!\big((\bm{\mathcal{J}}^e)^\top\bm{\mathcal{J}}^e\big)(\xi,\eta)>0,
\qquad
(\xi,\eta)\in(0,1)^2.
\end{equation}
Consequently, $\bm{\mathcal{J}}^e(\xi,\eta)$ has full column rank on the interior of the element, and the parametrization $\bm{x}^e$ is regular on $(0,1)^2$.
\end{theorem}

\begin{proof}
By Theorem~\ref{thm:gram-bernstein}, the Gram determinant admits the representation
\begin{equation}
\det\!\big((\bm{\mathcal{J}}^e)^\top\bm{\mathcal{J}}^e\big)(\xi,\eta)
=
\frac{N^e(\xi,\eta)}{(W^e(\xi,\eta))^8},
\end{equation}
where
\begin{equation}
N^e(\xi,\eta)
=
\sum_{r=0}^{8p-2}\sum_{s=0}^{8q-2}
D_{rs}^e\, B_r^{8p-2}(\xi) B_s^{8q-2}(\eta).
\label{eq:numerator-Ne}
\end{equation}

We first establish the positivity of the numerator $N^e(\xi,\eta)$. 
For any $\xi\in(0,1)$ and any $r=0,\dots,8p-2$, the univariate Bernstein polynomial
\begin{equation}
B_r^{8p-2}(\xi)=\binom{8p-2}{r}\xi^r(1-\xi)^{8p-2-r}
\end{equation}
is strictly positive. Likewise, for any $\eta\in(0,1)$ and any $s=0,\dots,8q-2$,
\begin{equation}
B_s^{8q-2}(\eta)=\binom{8q-2}{s}\eta^s(1-\eta)^{8q-2-s}
\end{equation}
is strictly positive. Hence, each tensor-product basis function
\begin{equation}
B_r^{8p-2}(\xi)B_s^{8q-2}(\eta)
\end{equation}
is strictly positive on $(0,1)^2$.

Since $D_{rs}^e\ge 0$ for all $(r,s)$ and at least one coefficient is strictly positive, the expansion \eqref{eq:numerator-Ne} is a nonnegative linear combination of strictly positive functions with at least one strictly positive weight. Therefore,
\begin{equation}
N^e(\xi,\eta)>0,
\qquad
(\xi,\eta)\in(0,1)^2.
\end{equation}

Next, by assumption, $W^e(\xi,\eta)>0$ on $[0,1]^2$, and hence
\begin{equation}
(W^e(\xi,\eta))^8>0
\qquad\text{on }[0,1]^2.
\end{equation}
It follows that
\begin{equation}
\det\!\big((\bm{\mathcal{J}}^e)^\top\bm{\mathcal{J}}^e\big)(\xi,\eta)>0,
\qquad
(\xi,\eta)\in(0,1)^2.
\end{equation}

Finally, since $(\bm{\mathcal{J}}^e)^\top\bm{\mathcal{J}}^e$ is a $2\times 2$ Gram matrix, 
the positivity of its determinant implies that it is positive definite. 
Therefore, the column vectors of $\bm{\mathcal{J}}^e$ are linearly independent, and
\begin{equation}
\operatorname{rank}(\bm{\mathcal{J}}^e(\xi,\eta))=2,
\qquad
(\xi,\eta)\in(0,1)^2.
\end{equation}
This establishes the regularity of the parametrization on the interior of the element.
\end{proof}

\begin{remark}
The condition in Corollary~\ref{cor:regularity} is sufficient but not necessary. The strict positivity of the polynomial $N^e(\xi,\eta)$ on $(0,1)^2$ does not require all Bernstein coefficients to be nonnegative. Indeed, cancellations among Bernstein basis functions may yield a strictly positive polynomial even when some coefficients $D_{rs}^e$ are negative.

Therefore, failure of the coefficient sign condition does not imply loss of regularity; it only indicates that regularity cannot be certified by this criterion. In practice, such cases may be further examined using subdivision or more refined tests.
\end{remark}

\subsection{Necessary sign conditions for bijectivity}

While the sufficient condition in Corollary~\ref{cor:regularity} provides a convenient certificate for regularity, it is generally not sharp. 
To efficiently detect invalid parametrizations, it is useful to derive simple \emph{necessary conditions} that must be satisfied by any injective mapping.

We first recall a classical result for non-rational B\'ezier surfaces.

\begin{theorem}[Necessary corner-sign condition]
\label{thm:corner}
Let a non-rational B\'ezier surface of degree $(m,n)$ be given by
\begin{equation}
\mathbf{S}(u,v)
=
\sum_{i=0}^{m}\sum_{j=0}^{n}
\mathbf{P}_{ij} B_i^m(u) B_j^n(v),
\qquad (u,v)\in[0,1]^2.
\end{equation}
Let
\begin{equation}
    \det\!\big((\bm{\mathcal{J}})^\top \bm{\mathcal{J}}\big)(u,v)
    =
    \sum_{i=0}^{4m-2}\sum_{j=0}^{4n-2}
    D_{ij} B_i^{4m-2}(u) B_j^{4n-2}(v)
\end{equation}
be its Bernstein representation. 

If the parametrization $\mathbf{S}$ is injective on $[0,1]^2$, then the four corner coefficients
\begin{equation}
    D_{0,0},\quad D_{4m-2,0},\quad D_{0,4n-2},\quad D_{4m-2,4n-2}
\end{equation}
must have the same (nonnegative) sign.
\end{theorem}

\begin{proof}
If $\mathbf{S}$ is injective and continuously differentiable, then it is locally invertible in the interior of the domain. 
Hence, the Jacobian matrix $\bm{\mathcal{J}}(u,v)$ has full column rank and
\begin{equation}
    \det\!\big((\bm{\mathcal{J}})^\top \bm{\mathcal{J}}\big)(u,v) > 0
    \qquad \text{for all } (u,v)\in(0,1)^2.
\end{equation}

By the endpoint interpolation property of Bernstein polynomials, the corner coefficients coincide with the values of the polynomial at the four corners:
\begin{equation}
    D_{0,0} = \det\!\big((\bm{\mathcal{J}})^\top \bm{\mathcal{J}}\big)(0,0), \quad
    D_{4m-2,0} = \det\!\big((\bm{\mathcal{J}})^\top \bm{\mathcal{J}}\big)(1,0),
\end{equation}
\begin{equation}
    D_{0,4n-2} = \det\!\big((\bm{\mathcal{J}})^\top \bm{\mathcal{J}}\big)(0,1), \quad
    D_{4m-2,4n-2} = \det\!\big((\bm{\mathcal{J}})^\top \bm{\mathcal{J}}\big)(1,1).
\end{equation}

If any of these values were negative while the determinant is strictly positive in the interior, 
then by continuity there would exist a point in the domain where the determinant vanishes, 
contradicting local invertibility. 
Therefore, all corner coefficients must be nonnegative and share the same sign.
\end{proof}

\begin{remark}
The corner coefficients represent the values of the Gram determinant at the four vertices of the parameter domain. The condition in Theorem~\ref{thm:corner} therefore admits a clear geometric interpretation: the local area distortion at the four corners must be consistent in sign. A violation of this condition indicates the presence of a fold or orientation reversal.

This condition is necessary but not sufficient. Even if all corner coefficients are nonnegative, the determinant may still vanish or change sign in the interior. Nevertheless, the test is extremely inexpensive and provides an effective first-stage filter for detecting invalid patches.
\end{remark}

For rational B\'ezier patches, the Gram determinant admits the representation derived in Theorem~\ref{thm:gram-bernstein}. Although the presence of the weight function introduces a rational denominator, the numerator
\begin{equation}
N^e(\xi,\eta)
=
\sum_{r,s} D_{rs}^e B_r(\xi) B_s(\eta)
\end{equation}
remains a polynomial in Bernstein form.

An analogous sign-consistency condition can therefore be imposed on the corner coefficients of $D_{rs}^e$. In particular, any inconsistency in their signs implies that the Gram determinant must vanish somewhere in the domain, and hence the parametrization cannot be injective.

This observation forms the basis of the rejection criterion used in the algorithm presented in the next subsection.

\subsection{Subdivision-based hierarchical checking}

When neither the sufficient condition in Corollary~\ref{cor:regularity} nor the necessary condition in Theorem~\ref{thm:corner} yields a conclusive result, we employ a subdivision-based hierarchical strategy to further analyze the parametrization.

Given a rational B\'ezier patch defined over $[0,1]^2$, the de Casteljau algorithm enables an exact tensor-product subdivision into four subpatches over smaller parametric domains. Each subpatch remains a rational B\'ezier surface of the same degree, with control points obtained through affine combinations in homogeneous coordinates. Importantly, this subdivision preserves the geometry exactly, such that the union of the subpatches represents the same surface as the original patch.

Let
\begin{equation}
N^e(\xi,\eta)
=
\sum_{r,s} D_{rs}^e \, B_r(\xi) B_s(\eta)
\end{equation}
denote the Bernstein representation of the numerator of the Gram determinant on a given patch. After subdivision, each subpatch admits its own Bernstein representation with coefficients $\widetilde{D}_{rs}$ obtained from $D_{rs}^e$ via repeated affine combinations. These refined coefficients encode increasingly localized information about the behavior of $N^e$ over the corresponding subdomains.

\begin{figure}
    \centering
    \subfigure[Original rational B\'ezier patch]{
        \includegraphics[width=0.45\linewidth]{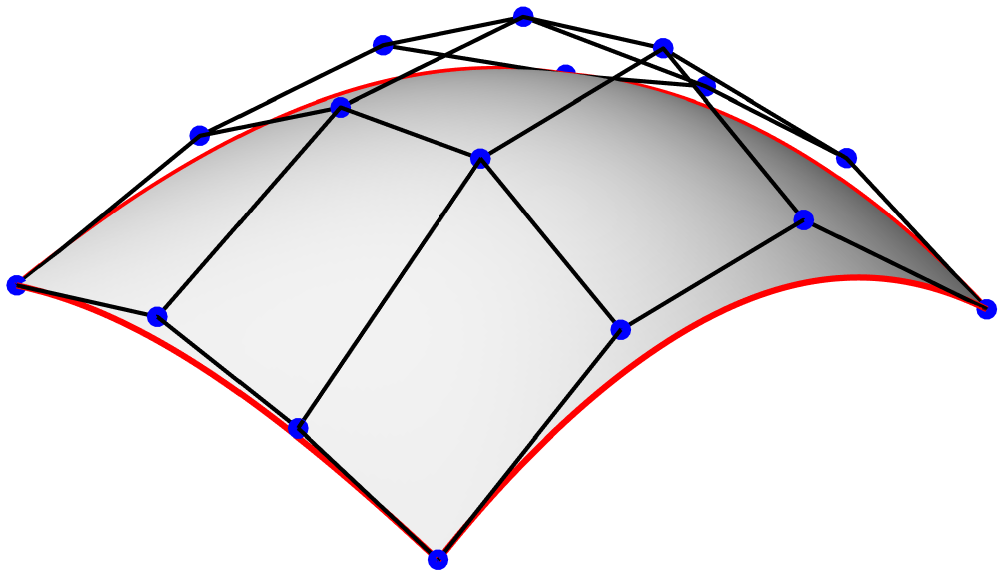}
        \label{fig:bezier_subdivision_original}}
    \quad
    \subfigure[Four subpatches after tensor-product subdivision]{
        \includegraphics[width=0.45\linewidth]{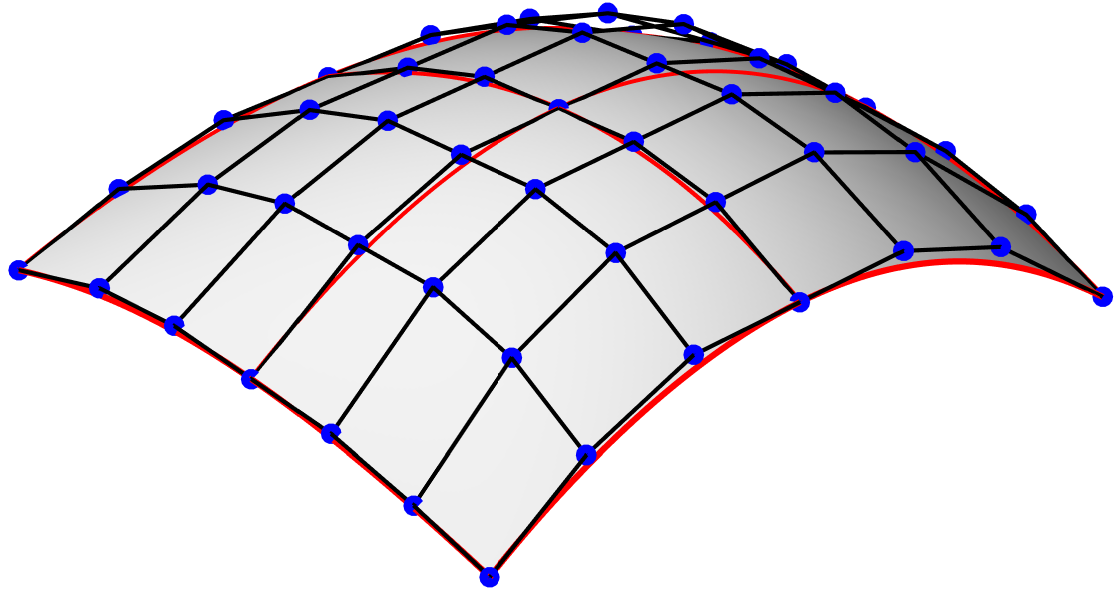}
        \label{fig:bezier_subdivision_result}}
    \caption{Tensor-product subdivision of a rational B\'ezier surface. 
    The original patch is subdivided at $u=0.5$ and $v=0.5$ in the parametric directions using the de Casteljau algorithm in homogeneous coordinates, producing four rational B\'ezier subpatches. 
    The subdivision preserves the exact geometry and forms the basis of the hierarchical refinement strategy for bijectivity analysis.}
    \label{fig:bezier_subdivision}
\end{figure}

Figure~\ref{fig:bezier_subdivision} illustrates the subdivision process. By applying the de Casteljau algorithm at $u=0.5$ and $v=0.5$ in homogeneous coordinates, the original patch is decomposed into four subpatches that preserve the exact geometry while enabling localized analysis.

The effectiveness of this approach relies on the convex hull property of Bernstein polynomials. As the parameter domain is subdivided, the Bernstein coefficients provide increasingly sharp bounds on the underlying function. In particular, if $N^e(\xi,\eta)$ is strictly positive but exhibits mixed-sign coefficients at a coarse level, then after sufficient subdivision the coefficients on subpatches tend to become sign-consistent. Conversely, if $N^e$ vanishes or changes sign, this behavior will be revealed at some refinement level through the appearance of non-positive coefficients.

Based on these properties, we adopt a hierarchical checking procedure. Starting from the initial set of B\'ezier patches, patches satisfying the sufficient condition are accepted, those violating the necessary condition are rejected, and the remaining patches are recursively subdivided and re-evaluated. This process naturally defines a tree structure over the parameter domain, where each level corresponds to a finer resolution.

The procedure terminates when all patches are classified, when an invalid patch is detected, or when a prescribed maximum refinement level is reached. Although termination cannot be guaranteed in all ambiguous cases, in practice only a small number of refinement levels is typically required to obtain a conclusive result.

The above hierarchical procedure is not merely heuristic. Its effectiveness can be theoretically justified for strictly regular patches. Intuitively, as the parameter domain is refined, the Bernstein coefficients provide increasingly accurate local bounds of the underlying function. Consequently, if the numerator of the Gram determinant is uniformly positive, the coefficients on sufficiently small subpatches must eventually become strictly positive. The following theorem formalizes this asymptotic sign consistency under subdivision.

\begin{theorem}[Asymptotic sign consistency under subdivision]
\label{thm:subdivision_sign_consistency}
Let $\mathbf{x}^e : [0,1]^2 \to \mathbb{R}^3$ be a rational B\'ezier patch with positive weights, and let
\[
N^e(\xi,\eta)
=
\sum_{r=0}^{m}\sum_{s=0}^{n} D_{rs}^e \, B_r^m(\xi) B_s^n(\eta)
\]
be the Bernstein representation of the numerator of the Gram determinant associated with $\mathbf{x}^e$.
Assume that $\mathbf{x}^e$ is regular on $[0,1]^2$, and that there exists a constant $\rho>0$ such that
\[
N^e(\xi,\eta)\ge \rho,
\qquad \forall (\xi,\eta)\in [0,1]^2.
\]

For each refinement level $\ell\in\mathbb{N}$, uniformly subdivide $[0,1]^2$ into $2^\ell\times 2^\ell$ subdomains, and let
\[
N^{e,\ell,\alpha}(\xi,\eta)
=
\sum_{r=0}^{m}\sum_{s=0}^{n} \widetilde D_{rs}^{e,\ell,\alpha}\, B_r^m(\xi) B_s^n(\eta)
\]
denote the Bernstein representation of the restriction of $N^e$ to a subpatch indexed by $\alpha$.
Then there exists $\ell_0\in\mathbb{N}$ such that for every $\ell\ge \ell_0$, every subpatch $\alpha$, and every pair $(r,s)$,
\[
\widetilde D_{rs}^{e,\ell,\alpha} > 0.
\]
In other words, after sufficiently many subdivision steps, all Bernstein coefficients on every subpatch become strictly positive.
\end{theorem}

\begin{proof}
Since $N^e$ is a polynomial function on the compact domain $[0,1]^2$, it is continuous. By assumption, there exists $\rho>0$ such that
\[
N^e(\xi,\eta)\ge \rho,
\qquad \forall (\xi,\eta)\in[0,1]^2.
\]

Let $\mathcal{B}^e$ denote the blossom of $N^e$, which is symmetric and multi-affine in its $m+n$ arguments and satisfies
\[
N^e(\xi,\eta)
=
\mathcal{B}^e(\underbrace{\xi,\ldots,\xi}_{m\ \text{times}};
\underbrace{\eta,\ldots,\eta}_{n\ \text{times}}).
\]
For a subdomain
\[
Q=[a,b]\times[c,d]\subset [0,1]^2,
\]
the Bernstein coefficients of the restriction of $N^e$ to $Q$ are given by blossom evaluations of the form
\[
\widetilde D_{rs}
=
\mathcal{B}^e(\underbrace{a,\ldots,a}_{m-r},
\underbrace{b,\ldots,b}_{r};
\underbrace{c,\ldots,c}_{n-s},
\underbrace{d,\ldots,d}_{s}).
\]

Now let $(\bar\xi,\bar\eta)$ be the midpoint of $Q$. Since the blossom $\mathcal{B}^e$ is continuous on the compact set $[0,1]^m\times[0,1]^n$, it is uniformly continuous. Therefore, for $\varepsilon=\rho/2$, there exists $\delta>0$ such that whenever all blossom arguments lie within distance $\delta$ of $(\bar\xi,\bar\eta)$, one has
\[
\left|
\widetilde D_{rs} - N^e(\bar\xi,\bar\eta)
\right|
< \frac{\rho}{2}.
\]

Under uniform subdivision, the diameter of each subdomain tends to zero as $\ell\to\infty$. Hence there exists $\ell_0$ such that for every $\ell\ge \ell_0$, every subdomain $Q$ at level $\ell$ has diameter smaller than $\delta$. It follows that for every coefficient on every such subdomain,
\[
\widetilde D_{rs}
\ge
N^e(\bar\xi,\bar\eta) - \frac{\rho}{2}
\ge
\rho - \frac{\rho}{2}
=
\frac{\rho}{2}
>0.
\]
Therefore, all Bernstein coefficients on every subpatch are strictly positive for all $\ell\ge \ell_0$.
\end{proof}

As a direct consequence, the subdivision-based hierarchical checking procedure admits a finite certification property for strictly regular T-spline patches.

\begin{corollary}[Finite certification for regular T-spline patches]
\label{cor:finite_certification}
Let a rational T-spline surface be decomposed by B\'ezier extraction into finitely many rational B\'ezier patches.
If every extracted patch is regular and its corresponding numerator $N^e$ satisfies
\[
N^e(\xi,\eta) > 0
\qquad \forall (\xi,\eta)\in[0,1]^2,
\]
then there exists a finite subdivision level such that the sufficient condition in Corollary~\ref{cor:regularity} is satisfied on every subdivided subpatch.
Consequently, the subdivision-based hierarchical checking procedure terminates after finitely many refinement steps and certifies regularity.
\end{corollary}

\subsection{Algorithm for bijectivity checking}

Based on the Bernstein representation derived in Theorem~\ref{thm:gram-bernstein}, together with the sufficient and necessary conditions established in the previous subsections, we propose a practical algorithm for checking the bijectivity of rational T-spline parametrizations.

The central idea is to classify each B\'ezier patch into three categories. A patch is declared \emph{valid} if it satisfies the sufficient condition in Corollary~\ref{cor:regularity}, and \emph{invalid} if it violates the necessary corner condition in Theorem~\ref{thm:corner}. Patches for which neither condition is conclusive are labeled \emph{undetermined} and are further analyzed by subdivision. This leads to a hierarchical decision process in which refinement is restricted to ambiguous regions of the parameter domain.

The algorithm combines three complementary mechanisms: a sufficient condition for immediate certification, a necessary condition for early rejection, and a subdivision strategy that progressively resolves ambiguous cases. As a result, it naturally defines a hierarchical tree over the parameter domain, in which only undecided regions are refined.

In practice, the computational cost is dominated by the evaluation of Bernstein coefficients and the subdivision process. However, most patches can typically be classified at a coarse level: patches with consistent coefficients are accepted immediately, while those violating the corner condition are rejected without further refinement. Consequently, only a small fraction of patches require subdivision, and the overall computational effort remains moderate even for complex geometries.

\begin{algorithm}
\caption{Hierarchical bijectivity checking for rational T-spline parametrizations}
\label{alg:bijectivity_check}
\begin{algorithmic}[1]
\STATE \textbf{Input:} Control points $\mathbf{P}_{ij}$, weights $\omega_{ij}$, degrees $p,q$, maximum refinement level $K_{\max}$
\STATE \textbf{Output:} Global status \texttt{Valid}, \texttt{Invalid}, or \texttt{Undetermined}

\STATE Perform B\'ezier extraction and construct the initial set of rational B\'ezier patches $\mathcal{E}^{(0)}$
\STATE Set $k \gets 0$

\WHILE{$k \le K_{\max}$}
    \STATE Initialize the unresolved patch set $\mathcal{U}^{(k)} \gets \emptyset$
    
    \FOR{each patch $e \in \mathcal{E}^{(k)}$}
        \STATE Compute the Bernstein coefficients $\{D_{rs}^e\}$ of the Gram determinant on $e$
        
        \IF{the corner coefficients of $e$ violate the sign-consistency condition}
            \RETURN \texttt{Invalid}
        \ELSIF{$D_{rs}^e \ge 0$ for all $(r,s)$ and $D_{rs}^e > 0$ for at least one $(r,s)$}
            \STATE Classify $e$ as \texttt{Valid}
        \ELSE
            \STATE Add $e$ to $\mathcal{U}^{(k)}$
        \ENDIF
    \ENDFOR
    
    \IF{$\mathcal{U}^{(k)} = \emptyset$}
        \RETURN \texttt{Valid}
    \ENDIF
    
    \IF{$k = K_{\max}$}
        \RETURN \texttt{Undetermined}
    \ENDIF
    
    \STATE Subdivide each patch in $\mathcal{U}^{(k)}$ to generate the refined patch set $\mathcal{E}^{(k+1)}$
    \STATE Set $k \gets k+1$
\ENDWHILE
\end{algorithmic}
\end{algorithm}

The procedure terminates when all patches are classified, when an invalid patch is detected, or when the maximum refinement level $K_{\max}$ is reached. In the latter case, the algorithm returns \texttt{Undetermined}. Although this situation cannot be excluded in theory, it is rarely encountered in practice, where only a few refinement levels are typically sufficient to reach a conclusive decision.

\section{Numerical experiments}
\label{sec:numerical}

In this section, we present a series of numerical experiments to assess the performance of the proposed bijectivity detection framework. The experiments are designed to validate the correctness of the Bernstein-coefficient representation, evaluate the effectiveness of the proposed sign conditions for classification, and examine the behavior of the subdivision-based hierarchical strategy in resolving ambiguous cases. In addition, we investigate the computational efficiency of the method and its applicability to practical T-spline geometries.

\subsection{Experimental setup}

All experiments are performed on a machine equipped with an Apple M1 Pro CPU and 16~GB of memory. The core numerical components, including B\'ezier extraction, Bernstein coefficient evaluation, and hierarchical subdivision, are implemented in \texttt{C++}, while visualization and post-processing are carried out in \texttt{MATLAB}.

All computations are conducted in double precision. Unless otherwise specified, the maximum subdivision level is set to $K_{\max} = 6$.

\subsection{Validation of the Bernstein-coefficient formulation}

We validate the Bernstein-coefficient representation derived in Theorem~\ref{thm:gram-bernstein} using representative B\'ezier patches, including both polynomial and rational cases.

\begin{table}
\centering
\caption{Control points of the polynomial B\'ezier patch (bi-degree $(3,3)$).}
\label{tab:control-points}
\begin{tabular}{c c c c c c c c c c c c c c c c}
\hline
Index & $x$ & $y$ & $z$ &
Index & $x$ & $y$ & $z$ &
Index & $x$ & $y$ & $z$ &
Index & $x$ & $y$ & $z$ \\
\hline
$(1,1)$ & 0   & 0   & 0 
& $(1,2)$ & 0   & 1/3 & 0 
& $(1,3)$ & 0   & 2/3 & 0 
& $(1,4)$ & 0   & 1   & 0 \\
$(2,1)$ & 1/3 & 0   & 0 
& $(2,2)$ & 1/3 & 1/3 & 1 
& $(2,3)$ & 1/3 & 2/3 & 1 
& $(2,4)$ & 1/3 & 1   & 0 \\
$(3,1)$ & 2/3 & 0   & 0 
& $(3,2)$ & 2/3 & 1/3 & 1 
& $(3,3)$ & 2/3 & 2/3 & 1 
& $(3,4)$ & 2/3 & 1   & 0 \\
$(4,1)$ & 1   & 0   & 0 
& $(4,2)$ & 1   & 1/3 & 0 
& $(4,3)$ & 1   & 2/3 & 0 
& $(4,4)$ & 1   & 1   & 0 \\
\hline
\end{tabular}
\end{table}

For the polynomial case, the control points are defined on a uniform $4\times4$ grid over $[0,1]^2$ (see Table~\ref{tab:control-points}), with all weights set to unity. For the rational case, a nontrivial geometry is constructed by introducing a nonlinear elevation and non-uniform weights:
\begin{equation}
z_{ij}
=
0.5\sin(2\pi x_i)\cos(2\pi y_j)
+ 0.5\,x_i y_j,
\end{equation}
\begin{equation}
w_{ij}
=
1 + 0.3(x_i + 0.5 y_j) + 0.2\sin(\pi x_i y_j).
\end{equation}

\begin{figure}
    \centering
    \subfigure[Polynomial B\'ezier surface]{
        \includegraphics[width=0.29\linewidth]{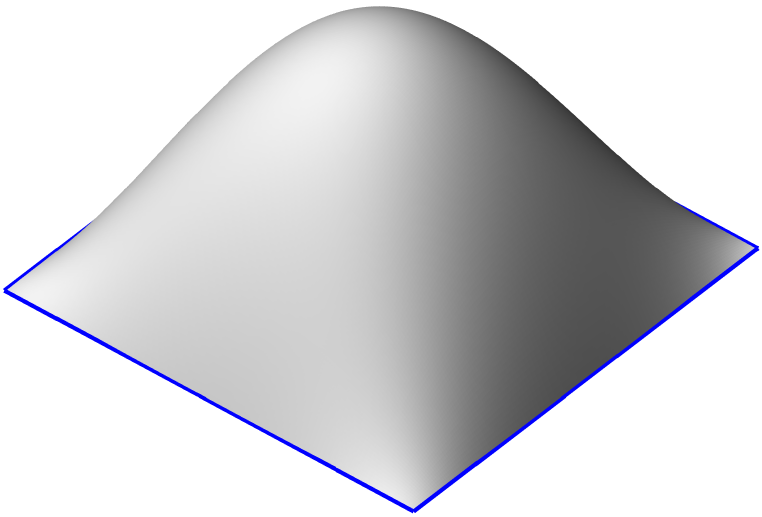}}
    \hfill
    \subfigure[Gram determinant (direct evaluation)]{
        \includegraphics[width=0.32\linewidth]{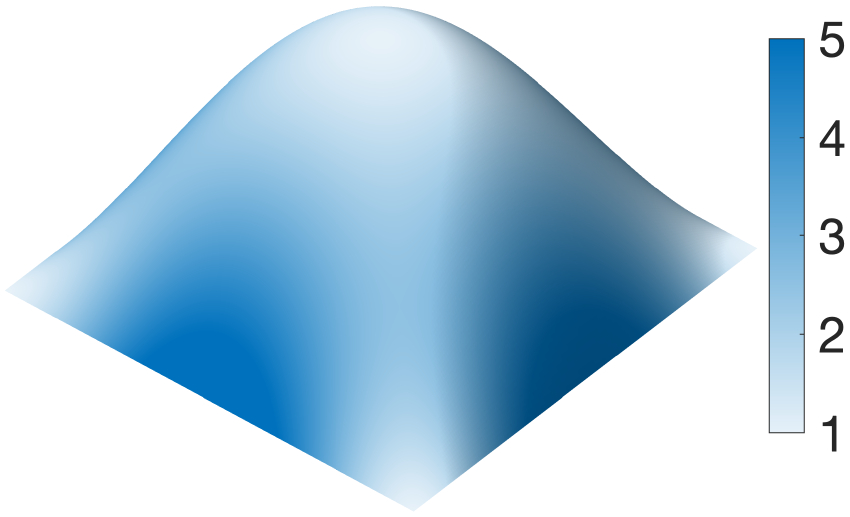}}
    \hfill
    \subfigure[Absolute error (polynomial case)]{
        \includegraphics[width=0.35\linewidth]{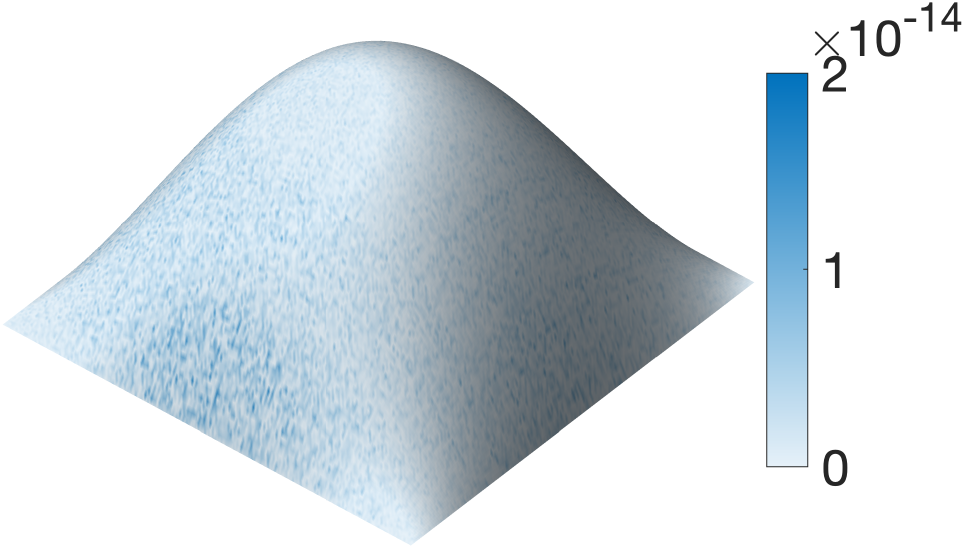}} \\
    \subfigure[Rational B\'ezier surface]{
        \includegraphics[width=0.29\linewidth]{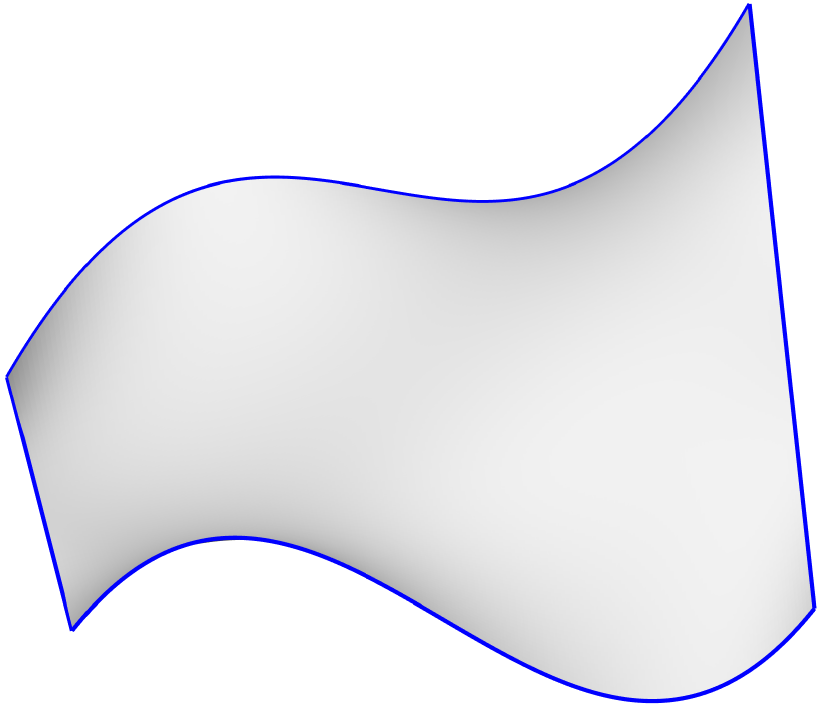}}
    \hfill
    \subfigure[Gram determinant (direct evaluation)]{
        \includegraphics[width=0.32\linewidth]{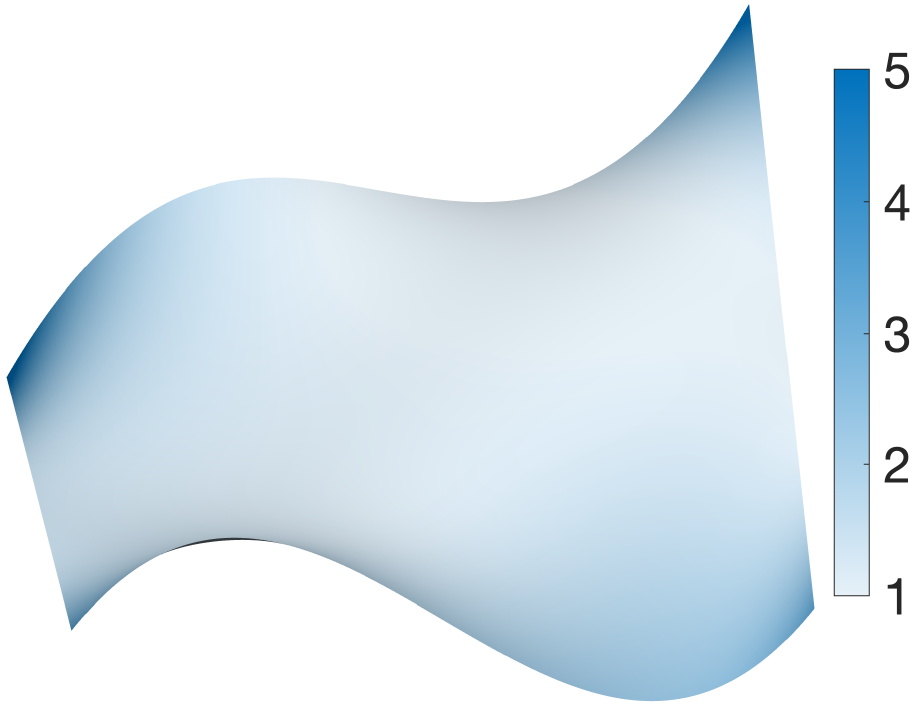}}
    \hfill
    \subfigure[Absolute error (rational case)]{
        \includegraphics[width=0.35\linewidth]{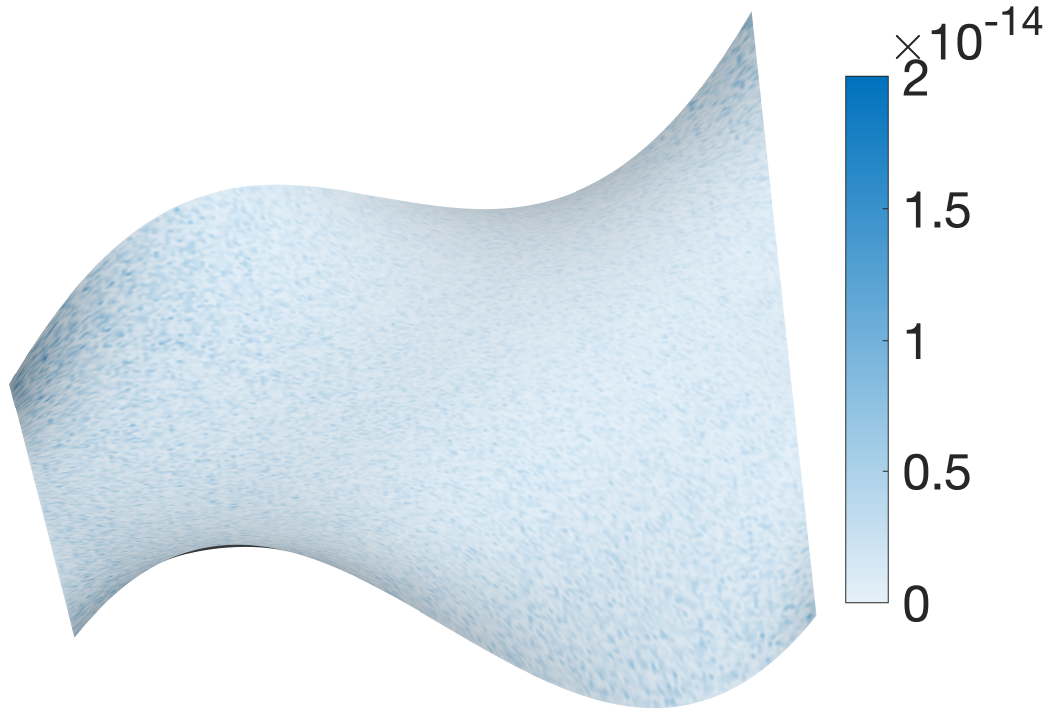}}
    \caption{Validation of the Bernstein-coefficient formulation for polynomial and rational B\'ezier patches. In each case, the Gram determinant obtained via direct numerical differentiation is compared with the Bernstein-based reconstruction. The first row corresponds to the polynomial case, and the second row to the rational case. The absolute errors are observed to be at the level of machine precision.}
    \label{fig:bernstein-validation}
\end{figure}

For each patch, the quantity $\det\!\big((\bm{\mathcal{J}})^\top\bm{\mathcal{J}}\big)$ is evaluated at uniformly sampled points in the parametric domain and compared with the Bernstein-based reconstruction. The accuracy is measured using the maximum absolute and relative errors.

The quantitative comparison confirms the accuracy of the proposed formulation. For the polynomial case, the maximum absolute and relative errors are $2.22\times10^{-14}$ and $4.64\times10^{-15}$, respectively, while for the rational case they are $1.51\times10^{-14}$ and $6.52\times10^{-15}$. In all cases, the discrepancies are at the level of machine precision, confirming the correctness of the Bernstein-based reconstruction.

\subsection{Regularity classification of subdivided rational B\'ezier patches}

We evaluate the proposed regularity checking framework on representative rational B\'ezier patches. The framework integrates sufficient and necessary sign conditions with a subdivision-based refinement strategy to provide a unified classification of patch regularity.

For a broad class of patches, regularity can be determined directly from the Bernstein coefficients, without requiring hierarchical refinement. If all coefficients $D_{rs}$ are nonnegative, the sufficient condition is satisfied and the patch is classified as \texttt{Valid}. Conversely, if the corner coefficients violate the sign-consistency condition in Theorem~\ref{thm:corner}, the patch is classified as \texttt{Invalid}.

These two cases represent situations where the regularity can be determined a priori from the control coefficients, providing an efficient and reliable classification mechanism. Representative examples are shown in Fig.~\ref{fig:direct_classification}, where both a regular patch and an invalid configuration are correctly identified without subdivision.

\begin{figure}
    \centering
    \subfigure[Regular patch with strictly positive Gram determinant]{\includegraphics[width=0.43\linewidth]{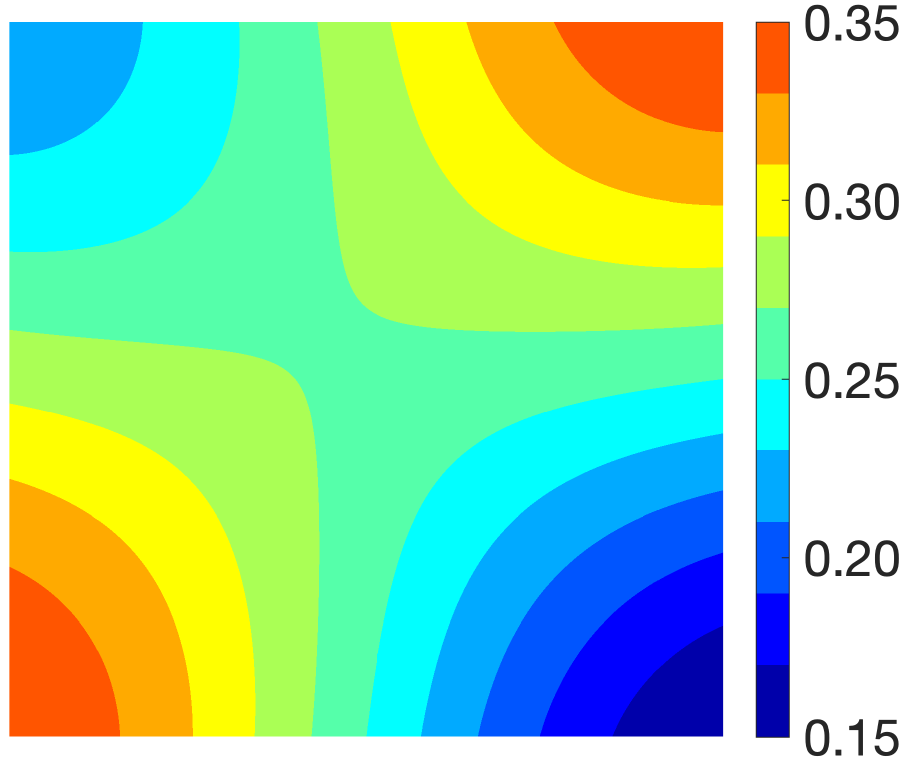}}
    \quad
    \subfigure[Invalid patch with sign inconsistency in corner coefficients]{\includegraphics[width=0.43\linewidth]{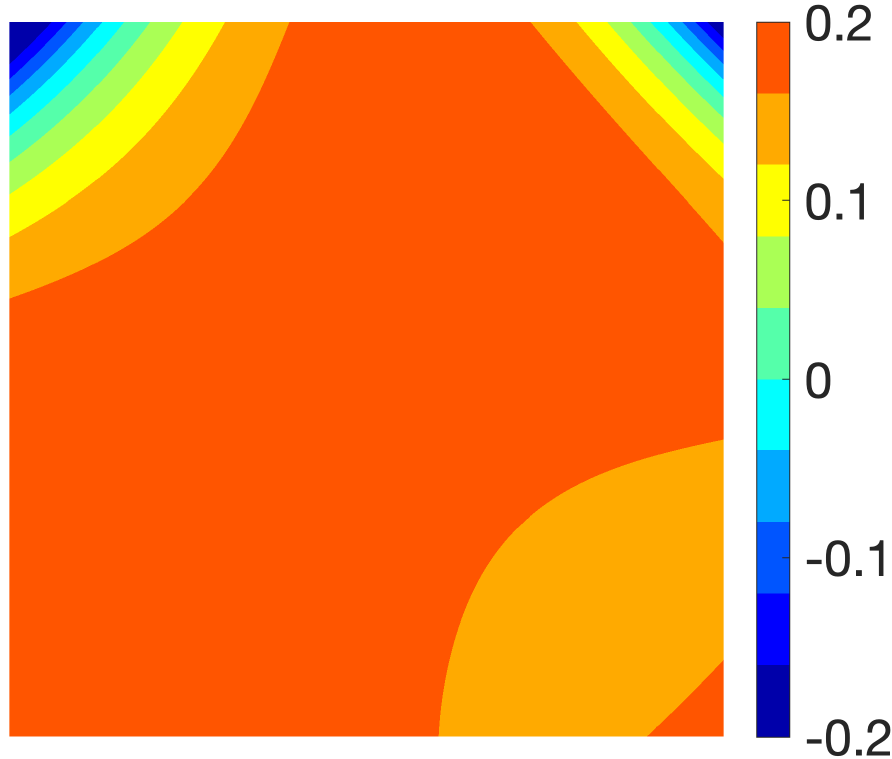}}
    \caption{Direct classification based on Bernstein coefficients. (a) A patch satisfying the sufficient condition, where the Gram determinant remains positive over the entire domain. (b) A patch violating the necessary corner condition, indicating a loss of regularity.}
    \label{fig:direct_classification}
\end{figure}

For patches that cannot be conclusively classified by the direct criteria, we employ a hierarchical subdivision strategy. Such ambiguous cases typically arise when the Bernstein coefficients exhibit mixed signs or when the Gram determinant becomes small in localized regions.

The effectiveness of this approach relies on the localizing property of the Bernstein representation. At a coarse level, the coefficients encode only global information and may obscure localized irregularities. Subdivision generates new sets of coefficients associated with smaller parametric subdomains, thereby revealing finer-scale variations of the Gram determinant. Consequently, regions where the determinant is close to zero or changes sign become progressively localized and easier to identify.

Figure~\ref{fig:subdivision_classification} illustrates this behavior. The initial patch exhibits mixed-sign or near-degenerate regions in the Gram determinant, preventing direct classification. Through hierarchical subdivision, these regions are progressively isolated, and the resulting subpatches can be clearly classified as \texttt{Valid} or \texttt{Invalid}.

Starting from an initially ambiguous patch, the classification is iteratively refined across subdivision levels. At each level, subpatches are categorized as \texttt{Valid}, \texttt{Invalid}, or \texttt{Undetermined}. In practice, the number of \texttt{Undetermined} patches decreases rapidly, leading to a definitive classification after only a few refinement steps.

\begin{figure}
    \centering
    \subfigure[Distribution of the Gram determinant for an ambiguous patch]{
        \includegraphics[width=0.47\linewidth]{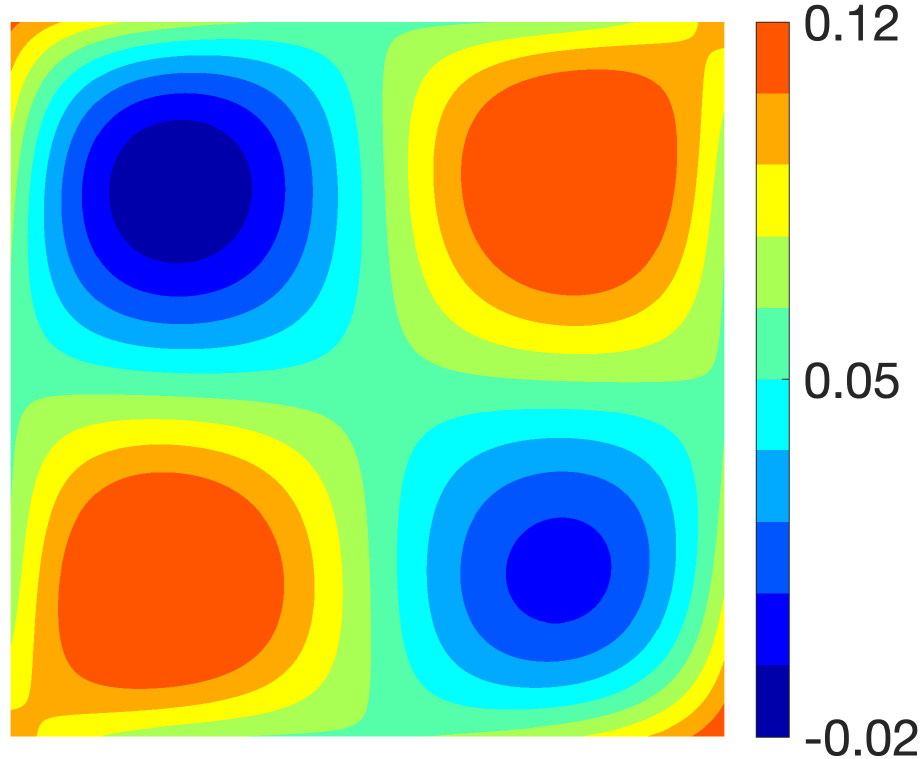}
    }
    \quad
    \subfigure[Hierarchical subdivision and classification of subpatches]{
        \includegraphics[width=0.38\linewidth]{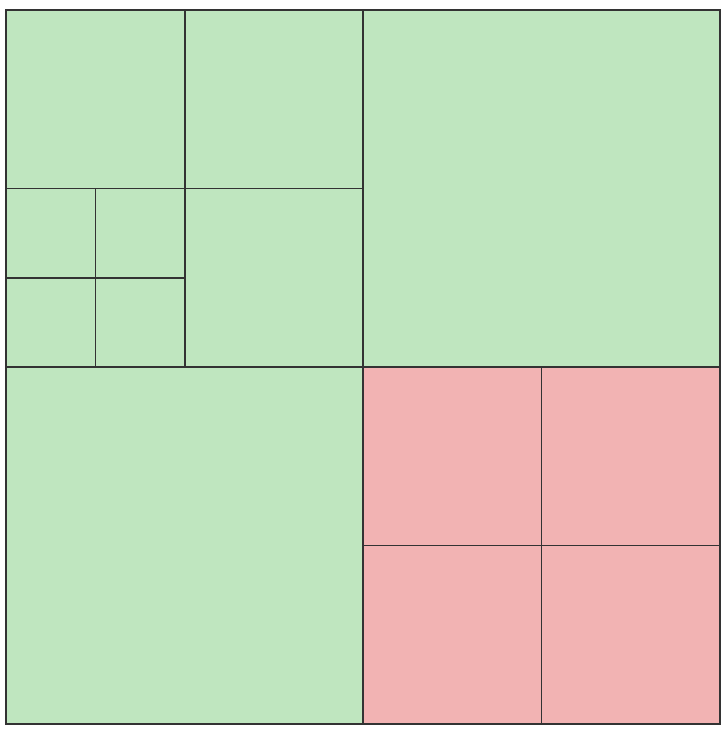}
    }
    \caption{Subdivision-based resolution of ambiguous cases. (a) The Gram determinant exhibits mixed-sign or near-zero regions, preventing direct classification. (b) The hierarchical subdivision progressively localizes these regions, leading to a clear classification of subpatches as \texttt{Valid} (green) or \texttt{Invalid} (red).
    }
    \label{fig:subdivision_classification}
\end{figure}

Table~\ref{tab:subdivision_stats} quantifies this process. Starting from a fully ambiguous configuration at Level~0, subdivision rapidly reduces the number of \texttt{Undetermined} patches, which vanish entirely by Level~3. This fast decay demonstrates that ambiguity is effectively resolved through recursive localization.

\begin{table}
\centering
\caption{Level-wise classification statistics during hierarchical subdivision.}
\label{tab:subdivision_stats}
\begin{tabular}{ccccc}
\hline
Level & Total & Valid & Invalid & Undetermined \\
\hline
0 & 1 & 0 & 0 & 1 \\
1 & 4 & 2 & 0 & 2 \\
2 & 8 & 3 & 4 & 1 \\
3 & 4 & 4 & 0 & 0 \\
\hline
\end{tabular}
\end{table}

\subsection{Computational efficiency}

We assess the computational performance of the proposed bijectivity checking framework with respect to the polynomial degree of the underlying representation.

We consider rational bi-degree-$(n,n)$ B\'ezier patches with $n=1,\dots,10$, and for each degree we evaluate $1000$ randomly generated samples. Figure~\ref{fig:computational-efficiency}(a) shows the average computation time as a function of the polynomial degree, while Table~\ref{tab:efficiency} reports detailed timing statistics, including mean, standard deviation, and extrema.

The results show a steady increase in runtime as the degree grows. This behavior can be attributed to the quadratic growth in the number of Bernstein coefficients with respect to $n$, which leads to a corresponding increase in algebraic operations such as inner products and polynomial multiplications. Nevertheless, the overall growth remains moderate, indicating that the method exhibits only mild sensitivity to the polynomial degree in practical settings.

To further assess robustness, we analyze the per-sample computation time over $1000$ independent runs. Figure~\ref{fig:computational-efficiency}(b) shows that the runtime remains highly stable, with only minor fluctuations across samples.

\begin{figure}
    \centering
    \subfigure[Computation time versus polynomial degree]{
        \includegraphics[width=0.47\linewidth]{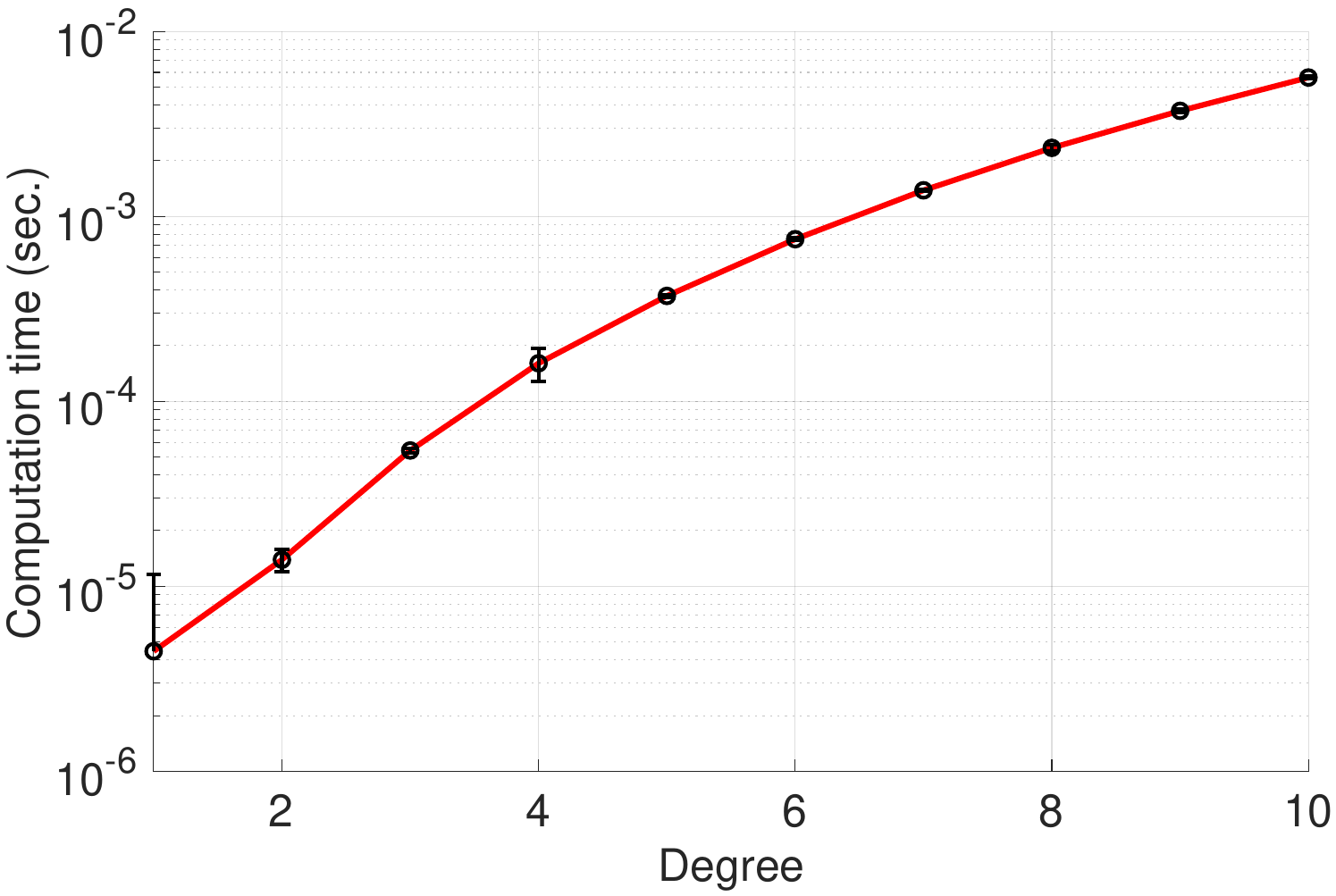}}
    \quad
    \subfigure[Per-sample computation time]{
        \includegraphics[width=0.47\linewidth]{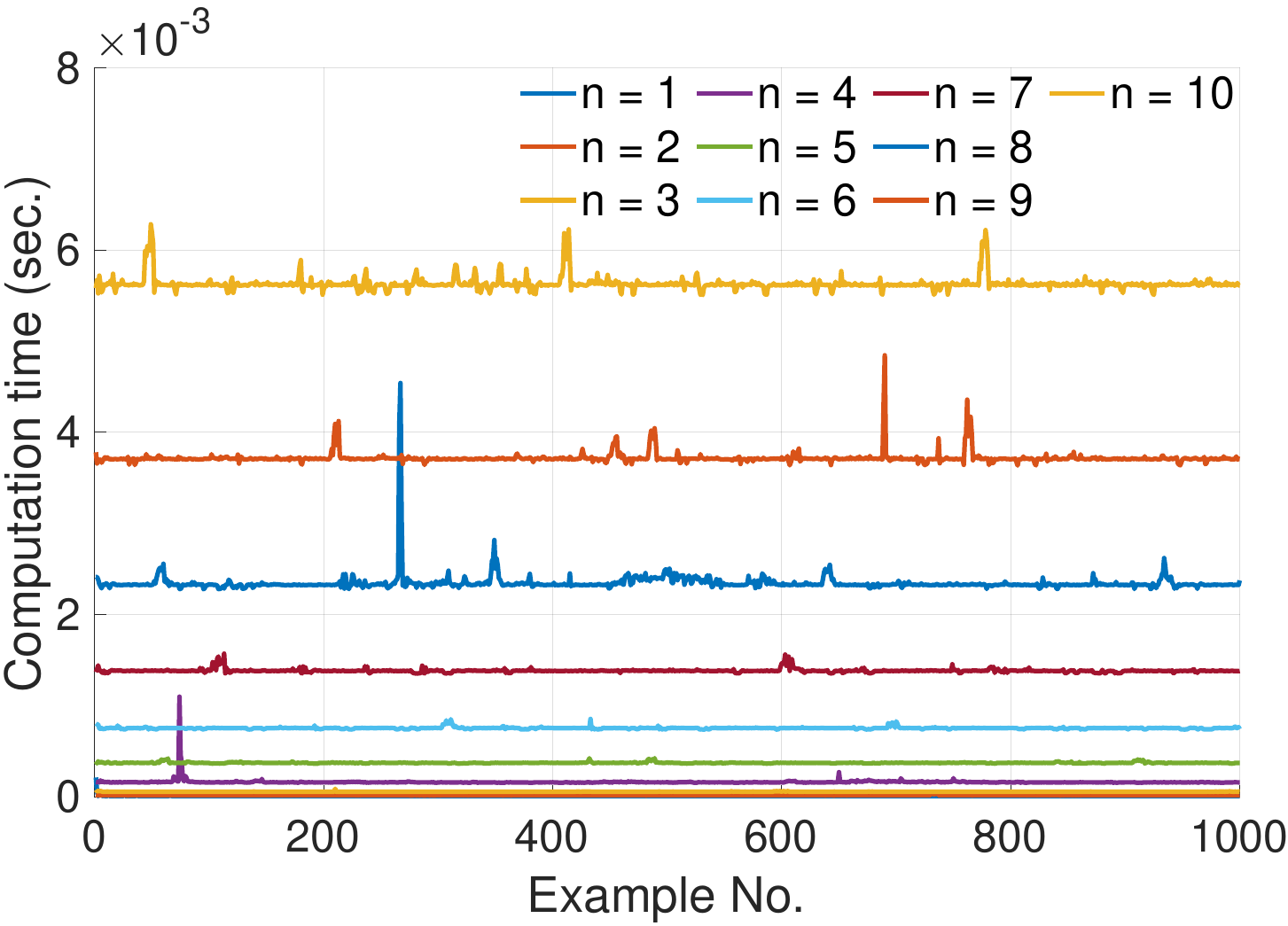}}
    \caption{Computational performance of the proposed coefficient evaluation algorithm. 
    (a) Average computation time as a function of the polynomial degree for rational bi-degree-$(n,n)$ B\'ezier patches with $n=1,\dots,10$, showing a steady increase consistent with the growth in the number of Bernstein coefficients. 
    (b) Per-sample computation time over 1000 random samples (for representative degrees), demonstrating highly stable runtimes with only minor fluctuations.}
    \label{fig:computational-efficiency}
\end{figure}

This stability reflects the deterministic structure of the coefficient evaluation procedure and the absence of iterative solvers or data-dependent branching, making the method well suited for reliable large-scale computations.

\begin{table}
\centering
\caption{Computational cost of the coefficient evaluation algorithm for rational bi-degree-$(n,n)$ B\'ezier patches. For each degree, $1000$ random samples are evaluated.}
\label{tab:efficiency}
\begin{tabular}{c c c c c c c}
\hline
Degree $n$ & \#CPs & Mean (s) & Std (s) & Min (s) & Max (s) & Repeats \\
\hline
1  & 4   & $4.45\times10^{-6}$ & $7.11\times10^{-6}$ & $3.80\times10^{-6}$ & $2.11\times10^{-4}$ & 59 \\
2  & 9   & $1.39\times10^{-5}$ & $1.98\times10^{-6}$ & $1.34\times10^{-5}$ & $5.22\times10^{-5}$ & 181 \\
3  & 16  & $5.42\times10^{-5}$ & $1.60\times10^{-6}$ & $5.31\times10^{-5}$ & $8.65\times10^{-5}$ & 63 \\
4  & 25  & $1.61\times10^{-4}$ & $3.26\times10^{-5}$ & $1.54\times10^{-4}$ & $1.10\times10^{-3}$ & 28 \\
5  & 36  & $3.71\times10^{-4}$ & $6.37\times10^{-6}$ & $3.63\times10^{-4}$ & $4.22\times10^{-4}$ & 13 \\
6  & 49  & $7.52\times10^{-4}$ & $1.03\times10^{-5}$ & $7.37\times10^{-4}$ & $8.53\times10^{-4}$ & 7 \\
7  & 64  & $1.38\times10^{-3}$ & $1.91\times10^{-5}$ & $1.35\times10^{-3}$ & $1.57\times10^{-3}$ & 3 \\
8  & 81  & $2.34\times10^{-3}$ & $1.02\times10^{-4}$ & $2.28\times10^{-3}$ & $4.54\times10^{-3}$ & 3 \\
9  & 100 & $3.72\times10^{-3}$ & $6.39\times10^{-5}$ & $3.63\times10^{-3}$ & $4.84\times10^{-3}$ & 2 \\
10 & 121 & $5.63\times10^{-3}$ & $7.50\times10^{-5}$ & $5.51\times10^{-3}$ & $6.28\times10^{-3}$ & 1 \\
\hline
\end{tabular}
\end{table}

Taken together, these results provide a clear picture of the computational behavior of the proposed method. The runtime increases moderately with the polynomial degree, primarily due to the quadratic growth in the number of Bernstein coefficients. At the same time, the low variance observed across samples indicates that the method is robust and largely insensitive to input variations. In practical applications, most patches can be classified at the initial level using the sufficient or necessary conditions, and subdivision is required only in a small number of geometrically challenging regions. Consequently, the proposed framework achieves a favorable balance between robustness and efficiency, making it well suited for large-scale T-spline models in isogeometric analysis.

\subsection{Complex T-spline geometries.}

We evaluate the proposed framework on two representative T-spline models of practical relevance, namely a multi-patch bicycle frame and a chair geometry. Both examples exhibit nontrivial topology and heterogeneous local refinement, making them suitable benchmarks for assessing robustness on complex geometries.

For each model, the T-spline representation is decomposed into a collection of rational B\'ezier patches, and the proposed coefficient-based bijectivity criterion is applied in a patch-wise manner. The results show that the vast majority of patches satisfy the sufficient condition directly, indicating that the parametrizations are locally regular over most of the domain. Only a small subset of patches exhibits near-degenerate behavior, as indicated by small values of the Gram determinant. These regions are automatically identified by the method and can be further analyzed or refined if necessary.

\begin{figure}
    \centering
    \subfigure[multi-patch bicycle frame model]{
        \includegraphics[width=0.4\linewidth]{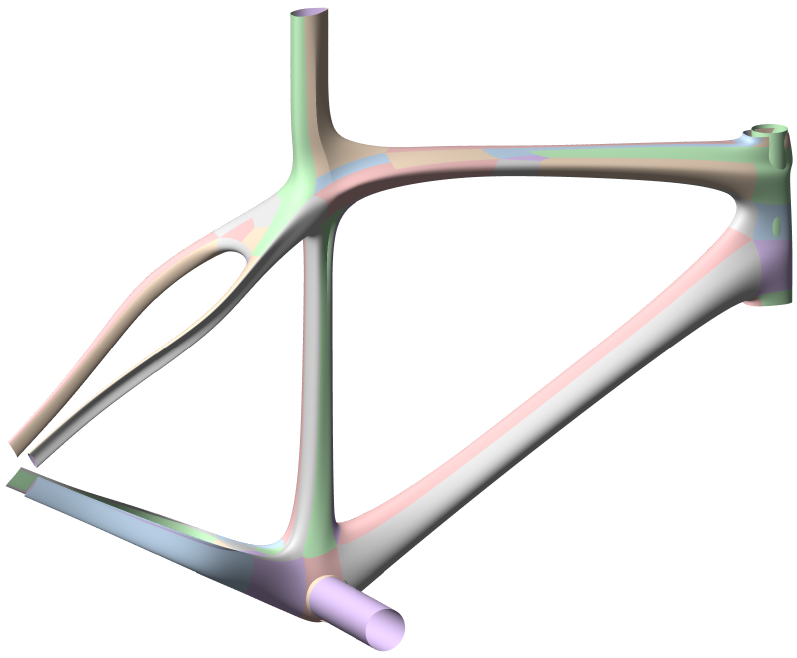}} \ 
    \subfigure[Distribution of the scaled Jacobian on the bicycle frame]{
        \includegraphics[width=0.46\linewidth]{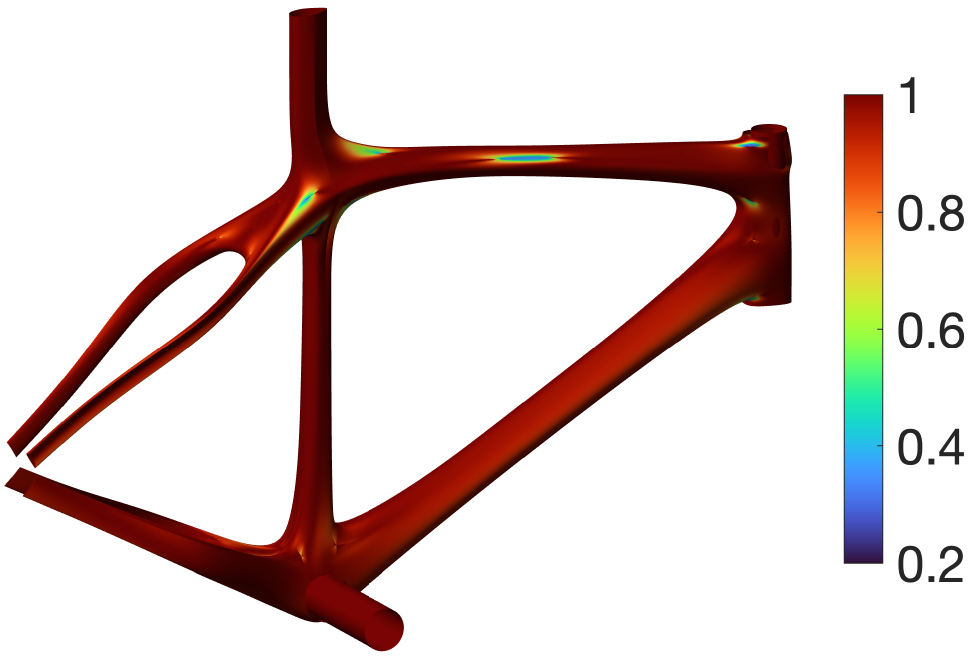}} \\
    \subfigure[multi-patch chair model]{
        \includegraphics[width=0.34\linewidth]{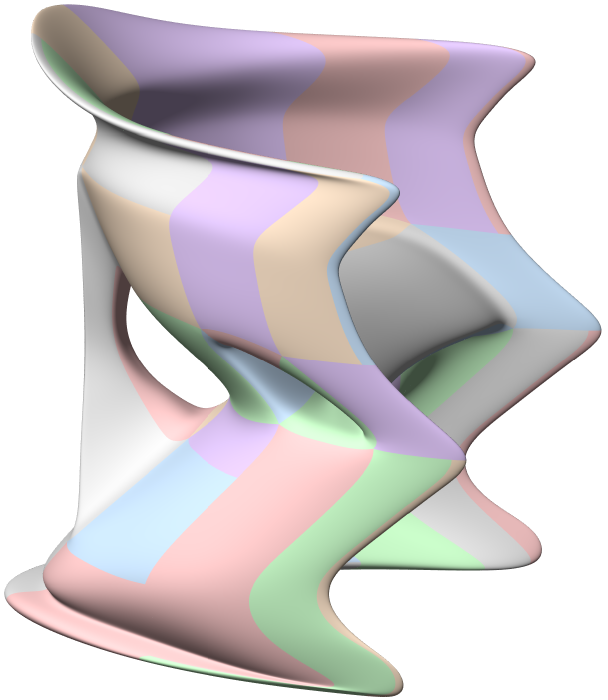}}
    \qquad
    \subfigure[Distribution of the scaled Jacobian on the chair model]{
        \includegraphics[width=0.39\linewidth]{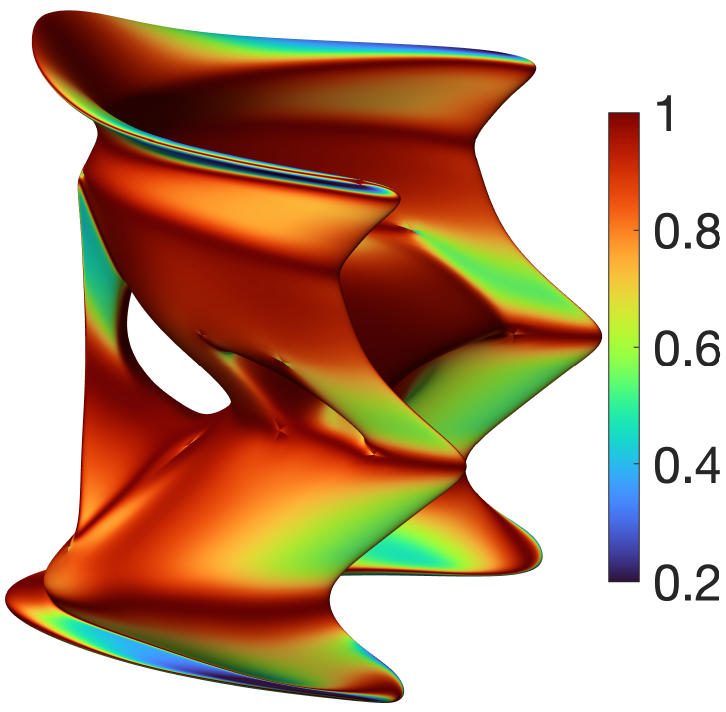}}
    \caption{Complex multi-patch T-spline geometries and the corresponding distributions of the scaled Jacobian. The left column shows the geometric models with patch-wise coloring, while the right column visualizes the spatial distribution of the scaled Jacobian. Lower values indicate increased geometric distortion and proximity to degeneracy.}
    \label{fig:real-models}
\end{figure}

Figure~\ref{fig:real-models} illustrates both the geometric models and the corresponding distributions of the scaled Jacobian. The left column shows the patch-wise geometry, while the right column visualizes the spatial variation of this normalized quantity. Although the classification is performed based on the Bernstein coefficients of the Gram determinant, the scaled Jacobian provides an intuitive visualization of geometric quality. Regions with lower scaled Jacobian values clearly correspond to areas of increased geometric distortion and potential loss of regularity. In particular, values approaching zero indicate near-degenerate parametrizations. Importantly, these regions remain spatially localized and do not propagate across the entire geometry, demonstrating that the proposed framework can effectively isolate problematic areas without triggering unnecessary global refinement.

From a computational perspective, the bicycle frame model is decomposed into $7372$ bi-cubic rational B\'ezier patches, with a total computation time of $0.732$ seconds. The chair model consists of $13180$ bi-cubic rational B\'ezier patches and requires $1.2405$ seconds. These results indicate that the computational cost scales approximately linearly with the number of patches, highlighting the efficiency of the approach for large-scale models.

The experiments demonstrate that the proposed framework is applicable to complex, real-world T-spline geometries. It provides a robust and efficient tool for assessing local regularity, detecting near-degenerate regions, and supporting scalable analysis in isogeometric modeling.

\section{Conclusion and outlook}
\label{sec:conclusion}

In this work, we have developed a Bernstein-based framework for the bijectivity analysis of rational T-spline surfaces. By exploiting B\'ezier extraction, the problem is reformulated at the element level, where the Gram determinant admits an explicit Bernstein representation. This formulation enables a rigorous and computable characterization of local regularity in terms of coefficient signs, thereby providing a practical pathway toward certifying global bijectivity.

The proposed approach combines coefficient-based sign conditions with a hierarchical subdivision strategy, yielding an adaptive procedure that efficiently resolves inconclusive regions. The proposed method offers a certified alternative that avoids dense evaluations while remaining fully compatible with standard isogeometric analysis workflows. More broadly, the framework highlights the fundamental role of Bernstein representations in geometric validity analysis and provides a principled basis for reliable parametrization in simulation-driven applications. Numerical experiments confirm that the method robustly distinguishes valid, invalid, and near-degenerate configurations, while maintaining high computational efficiency through localized refinement.

Several directions for future research remain open. A natural extension concerns volumetric parametrizations, such as trivariate T-splines, where bijectivity constraints are substantially more stringent. From a theoretical perspective, reducing the gap between sufficient and necessary conditions remains an important step toward a complete characterization. On the algorithmic side, coupling bijectivity verification with adaptive refinement or optimization could enable the automatic construction of valid parametrizations. Finally, extending the framework to geometries with extraordinary points and complex topologies calls for deeper analytical tools.

\section*{Acknowledgements}
\label{sec:Acknowledgement}
This research was supported by the National Natural Science Foundation of China (Nos.~12301490, 12471358), the Youth Talent Innovation Program of Dalian (No.~2023RQ088), and the China Postdoctoral Science Foundation (No.~2025M773071).

\bibliographystyle{elsarticle-num}
\bibliography{mybibfile}

\appendix
\section{Proof of Theorem~\ref{thm:gram-bernstein}}
\label{sec:appendix_A}

\begin{proof}
The result relies on the closure of the Bernstein basis under differentiation, multiplication, and linear combinations. Specifically, derivatives of Bernstein polynomials remain within a Bernstein basis of reduced degree, while products of Bernstein polynomials can be represented exactly in a Bernstein basis of higher degree. Consequently, all quantities involved in the Gram determinant admit Bernstein representations. The proof proceeds in several steps.

\textbf{Step 1: Rational derivative representation.}

For the rational B\'ezier parametrization
\begin{equation}
    \bm{x}^e = \frac{\mathbf{F}^e}{W^e},
\end{equation}
the partial derivatives are given by
\begin{equation}
\bm{x}^e_\xi = \frac{\mathbf{F}^e_\xi W^e - \mathbf{F}^e W^e_\xi}{(W^e)^2}
= \frac{\mathbf{U}^e}{(W^e)^2},
\end{equation}
and
\begin{equation}
    \bm{x}^e_\eta = \frac{\mathbf{F}^e_\eta W^e - \mathbf{F}^e W^e_\eta}{(W^e)^2} = \frac{\mathbf{V}^e}{(W^e)^2},
\end{equation}
where $\mathbf{U}^e$ and $\mathbf{V}^e$ are polynomial vector fields.

\textbf{Step 2: Bernstein structure of $\mathbf{U}^e$ and $\mathbf{V}^e$.}

We denote by $\mathcal{B}^{(p,q)}$ the tensor-product Bernstein space of bi-degree $(p,q)$. 
Let
\begin{equation}
\mathbf{F}^e(\xi,\eta)
=
\sum_{i=0}^{p}\sum_{j=0}^{q}
\mathbf{F}_{ij} \, B_i^p(\xi) B_j^q(\eta),
\qquad
W^e(\xi,\eta)
=
\sum_{k=0}^{p}\sum_{\ell=0}^{q}
W_{k\ell} \, B_k^p(\xi) B_\ell^q(\eta).
\end{equation}

Using standard differentiation formulas of Bernstein polynomials, we obtain
\begin{equation}
\mathbf{F}^e_\xi(\xi,\eta)
=
\sum_{i=0}^{p-1}\sum_{j=0}^{q}
\mathbf{F}^\xi_{ij} \, B_i^{p-1}(\xi) B_j^q(\eta),
\end{equation}
\begin{equation}
W^e_\xi(\xi,\eta)
=
\sum_{k=0}^{p-1}\sum_{\ell=0}^{q}
W^\xi_{k\ell} \, B_k^{p-1}(\xi) B_\ell^q(\eta),
\end{equation}
where the derivative coefficients are given by
\begin{equation}
\mathbf{F}^\xi_{ij}
=
p \left( \mathbf{F}_{i+1,j} - \mathbf{F}_{i,j} \right),
\qquad
W^\xi_{k\ell}
=
p \left( W_{k+1,\ell} - W_{k,\ell} \right).
\end{equation}

Using the Bernstein product rule, the product $\mathbf{F}^e_\xi W^e$ admits the expansion
\begin{equation}
\mathbf{F}^e_\xi W^e
=
\sum_{r=0}^{2p-1}\sum_{s=0}^{2q}
\left(
\sum_{\substack{i+k=r \\ j+\ell=s}}
\alpha_{i,k}^{p-1,p} \, \alpha_{j,\ell}^{q,q}
\, \mathbf{F}^\xi_{ij} \, W_{k\ell}
\right)
B_r^{2p-1}(\xi) B_s^{2q}(\eta),
\end{equation}
where
\begin{equation}
\alpha_{i,k}^{n,m}
=
\frac{\binom{n}{i}\binom{m}{k}}{\binom{n+m}{i+k}}.
\end{equation}

Similarly,
\begin{equation}
\mathbf{F}^e W^e_\xi
=
\sum_{r=0}^{2p-1}\sum_{s=0}^{2q}
\left(
\sum_{\substack{i+k=r \\ j+\ell=s}}
\alpha_{i,k}^{p,p-1} \, \alpha_{j,\ell}^{q,q}
\, \mathbf{F}_{ij} \, W^\xi_{k\ell}
\right)
B_r^{2p-1}(\xi) B_s^{2q}(\eta).
\end{equation}

Therefore, $\mathbf{U}^e = \mathbf{F}^e_\xi W^e - \mathbf{F}^e W^e_\xi=\sum_{r=0}^{2p-1}\sum_{s=0}^{2q}\mathbf{U}^e_{rs}B_r^{2p-1}(\xi) B_s^{2q}(\eta) \in \mathcal{B}^{(2p-1,2q)}$, and  the coefficients $\mathbf{U}^e_{rs}$ are given by
\begin{equation}
\mathbf{U}^e_{rs}
=
\sum_{\substack{i+k=r \\ j+\ell=s}}
\left(
\alpha_{i,k}^{p-1,p} \alpha_{j,\ell}^{q,q} \, \mathbf{F}^\xi_{ij} \, W_{k\ell}
-
\alpha_{i,k}^{p,p-1} \alpha_{j,\ell}^{q,q} \, \mathbf{F}_{ij} \, W^\xi_{k\ell}
\right).
\end{equation}

Similarly, for the derivative with respect to $\eta$, we write
\begin{equation}
\mathbf{F}^e_\eta
=
\sum_{i=0}^{p}\sum_{j=0}^{q-1}
\mathbf{F}^\eta_{ij} \, B_i^{p}(\xi) B_j^{q-1}(\eta),
\qquad
\mathbf{F}^\eta_{ij}
=
q \left( \mathbf{F}_{i,j+1} - \mathbf{F}_{i,j} \right),
\end{equation}
and
\begin{equation}
W^e_\eta
=
\sum_{k=0}^{p}\sum_{\ell=0}^{q-1}
W^\eta_{k\ell} \, B_k^{p}(\xi) B_\ell^{q-1}(\eta),
\qquad
W^\eta_{k\ell}
=
q \left( W_{k,\ell+1} - W_{k,\ell} \right).
\end{equation}

Applying the Bernstein product rule to
$\mathbf{V}^e=\mathbf{F}^e_\eta W^e-\mathbf{F}^e W^e_\eta$,
we obtain
\begin{equation}
\mathbf{V}^e
=
\sum_{r=0}^{2p}\sum_{s=0}^{2q-1}
\mathbf{V}^e_{rs} \, B_r^{2p}(\xi) B_s^{2q-1}(\eta),
\end{equation}
where the coefficients $\mathbf{V}^e_{rs}$ are given by
\begin{equation}
\mathbf{V}^e_{rs}
=
\sum_{\substack{i+k=r \\ j+\ell=s}}
\frac{\binom{p}{i}\binom{p}{k}}{\binom{2p}{r}}
\left[
\mathbf{F}^\eta_{ij} W_{k\ell}
\frac{\binom{q-1}{j}\binom{q}{\ell}}{\binom{2q-1}{s}}
-
\mathbf{F}_{ij} W^\eta_{k\ell}
\frac{\binom{q}{j}\binom{q-1}{\ell}}{\binom{2q-1}{s}}
\right].
\end{equation}

\textbf{Step 3: Gram determinant structure.}

The Jacobian matrix of the parametrization can be written as
\begin{equation}
\bm{\mathcal{J}}^e
=
\frac{1}{(W^e)^2}
[\mathbf{U}^e \ \mathbf{V}^e].
\end{equation}
Consequently, the associated Gram matrix is given by
\begin{equation}
(\bm{\mathcal{J}}^e)^\top \bm{\mathcal{J}}^e
=
\frac{1}{(W^e)^4}
\begin{bmatrix}
\mathbf{U}^e \!\cdot\! \mathbf{U}^e & \mathbf{U}^e \!\cdot\! \mathbf{V}^e \\
\mathbf{V}^e \!\cdot\! \mathbf{U}^e & \mathbf{V}^e \!\cdot\! \mathbf{V}^e
\end{bmatrix}.
\end{equation}
Taking the determinant yields
\begin{equation}
\det\!\big((\bm{\mathcal{J}}^e)^\top \bm{\mathcal{J}}^e\big)
=
\frac{1}{(W^e)^8}
\left[
(\mathbf{U}^e\!\cdot\!\mathbf{U}^e)
(\mathbf{V}^e\!\cdot\!\mathbf{V}^e)
-
(\mathbf{U}^e\!\cdot\!\mathbf{V}^e)^2
\right],
\label{eq:gram-structure}
\end{equation}
where we used the standard identity for the determinant of a $2\times2$ symmetric matrix.

\textbf{Step 4: Closure under inner products.}

The inner product of vector-valued Bernstein polynomials is defined componentwise. 
Since each component of $\mathbf{U}^e$ and $\mathbf{V}^e$ admits a Bernstein expansion, their inner products remain Bernstein polynomials with degrees obtained by addition:
\begin{equation}
\mathbf{U}^e\cdot\mathbf{U}^e \in \mathcal{B}^{(4p-2,\,4q)},
\qquad
\mathbf{V}^e\cdot\mathbf{V}^e \in \mathcal{B}^{(4p,\,4q-2)},
\end{equation}
\begin{equation}
\mathbf{U}^e\cdot\mathbf{V}^e \in \mathcal{B}^{(4p-1,\,4q-1)}.
\end{equation}

More precisely, let
\begin{equation}
\mathbf{U}^e
=
\sum_{i=0}^{2p-1}\sum_{j=0}^{2q}
\mathbf{U}^e_{ij} \, B_i^{2p-1}(\xi) B_j^{2q}(\eta),
\qquad
\mathbf{V}^e
=
\sum_{k=0}^{2p}\sum_{\ell=0}^{2q-1}
\mathbf{V}^e_{k\ell} \, B_k^{2p}(\xi) B_\ell^{2q-1}(\eta).
\end{equation}

Then, applying the Bernstein product rule componentwise, we obtain
\begin{equation}
\mathbf{U}^e \cdot \mathbf{U}^e
=
\sum_{r=0}^{4p-2}\sum_{s=0}^{4q}
A_{rs}^e \, B_r^{4p-2}(\xi) B_s^{4q}(\eta),
\end{equation}
with coefficients
\begin{equation}
A_{rs}^e
=
\sum_{\substack{i+i'=r \\ j+j'=s}}
(\mathbf{U}^e_{ij} \cdot \mathbf{U}^e_{i'j'})
\,
\frac{\binom{2p-1}{i}\binom{2p-1}{i'}}{\binom{4p-2}{r}}
\frac{\binom{2q}{j}\binom{2q}{j'}}{\binom{4q}{s}}.
\end{equation}

Similarly,
\begin{equation}
\mathbf{V}^e \cdot \mathbf{V}^e
=
\sum_{r=0}^{4p}\sum_{s=0}^{4q-2}
B_{rs}^e \, B_r^{4p}(\xi) B_s^{4q-2}(\eta),
\end{equation}
\begin{equation}
\mathbf{U}^e \cdot \mathbf{V}^e
=
\sum_{r=0}^{4p-1}\sum_{s=0}^{4q-1}
C_{rs}^e \, B_r^{4p-1}(\xi) B_s^{4q-1}(\eta),
\end{equation}
where the coefficients $B_{rs}^e$ and $C_{rs}^e$ are defined analogously as
\begin{equation}
\begin{aligned}
     B^e_{rs}&=\sum_{\substack{k+k'=r\\l+l'=s}}(\mathbf{V}^e_{kl}\cdot\mathbf{V}^e_{k'l'})\frac{\binom{2p}{k}\binom{2p}{k'}}{\binom{4p}{r}}\frac{\binom{2q-1}{l}\binom{2q-1}{l'}}{\binom{4q-2}{s}},\\
     C^e_{rs}&=\sum_{\substack{i+k=r\\j+l=s}}(\mathbf{U}^e_{ij}\cdot\mathbf{V}^e_{kl})\frac{\binom{2p-1}{i}\binom{2p}{k}}{\binom{4p-1}{r}}\frac{\binom{2q}{j}\binom{2q-1}{l}}{\binom{4q-1}{s}}.
\end{aligned}
\end{equation}

Applying the Bernstein product rule once more, we obtain
\begin{equation}
(\mathbf{U}^e\cdot\mathbf{U}^e)(\mathbf{V}^e\cdot\mathbf{V}^e)
=
\sum_{r=0}^{8p-2}\sum_{s=0}^{8q-2}
E_{rs}^e \, B_r^{8p-2}(\xi) B_s^{8q-2}(\eta),
\end{equation}
\begin{equation}
(\mathbf{U}^e\cdot\mathbf{V}^e)^2
=
\sum_{r=0}^{8p-2}\sum_{s=0}^{8q-2}
F_{rs}^e \, B_r^{8p-2}(\xi) B_s^{8q-2}(\eta).
\end{equation}
where 
\begin{equation}
\begin{aligned}
    E^e_{rs}&=\sum_{\substack{r_1+r_2=r\\s_1+s_2=s}}(A^e_{r_1s_1}\cdot B^e_{r_2s_2})\frac{\binom{4p-2}{r_1}\binom{4p}{r_2}}{\binom{8p-2}{r}}\frac{\binom{4q}{s_1}\binom{4q-2}{s_2}}{\binom{8q-2}{s}},\\
    F^e_{rs}&=\sum_{\substack{r_1+r_2=r\\s_1+s_2=s}}(C^e_{r_1s_1}\cdot C^e_{r_2})\frac{\binom{4p-1}{r_1}\binom{4p-1}{r_2}}{\binom{8p-2}{r}}\frac{\binom{4q-1}{s_1}\binom{4q-1}{s_2}}{\binom{8q-2}{s}}.
\end{aligned}  
\end{equation}

\textbf{Step 5: Final Bernstein expansion.}

Therefore, the numerator in \eqref{eq:gram-structure} admits the Bernstein expansion
\begin{equation}
N^e(\xi,\eta)
=
\sum_{r=0}^{8p-2}\sum_{s=0}^{8q-2}
D_{rs}^e \, B_r^{8p-2}(\xi) B_s^{8q-2}(\eta),
\label{eq:bernstein_coefficients}
\end{equation}
with coefficients
\begin{equation}
D_{rs}^e = E_{rs}^e - F_{rs}^e,
\end{equation}
which leads to the representation \eqref{eq:JeTJe-final}.
\end{proof}

\end{document}